\documentclass[reqno,11pt]{amsart}
\usepackage{amssymb,nicefrac,cite,bm}
\usepackage[mathscr]{eucal}
\usepackage[active]{srcltx}
\usepackage{dsfont}
\usepackage[symbol]{footmisc}
\usepackage{hyperref}
\usepackage{color}
\definecolor{dmagenta}{rgb}{.4,.1,.5}
\definecolor{dred}{rgb}{.6,.0,.0}
\definecolor{dgreen}{rgb}{.0,.5,.0}
\definecolor{dblue}{rgb}{.0,.0,.7}
\definecolor{violet}{rgb}{.3,.0,.9}
\definecolor{orange}{cmyk}{0,.5,.1,.0}
\definecolor{dcyan}{cmyk}{.5,.0,.0,.0}
\definecolor{dyellow}{cmyk}{.0,.0,.5,.0}

\numberwithin{equation}{section}

\newtheorem{theorem}{Theorem}[section]
\newtheorem{lemma}{Lemma}[section]

\theoremstyle{definition}

\newtheorem{assumption}{Assumption}
\theoremstyle{remark}
\newtheorem{remark}{Remark}[section]

\newcommand{\df}{\;{:=}\;}
\newcommand{\transp}{^{\mathsf{T}}}
\DeclareMathOperator*{\trace}{tr}

\newcommand{\sV}{\mathcal{V}}%Lyapunov

\newcommand{\RR}{\mathbb{R}}
\newcommand{\Rd}{\mathbb{R}^d}
\newcommand{\NN}{\mathbb{N}}

\newcommand{\abs}[1]{\lvert#1\rvert}

\newcommand{\norm}[1]{\lVert#1\rVert}

\newcommand{\Ind}{\mathds{1}}   % indicator function
\newcommand{\grad}{\nabla}

\begin{document}

\title[Risk-Sensitive SDG]{Nonzero-Sum Risk-Sensitive Stochastic Differential Games: A Multi-parameter Eigenvalue Problem Approach}
%\thanks{*This paper is dedicated to the memory of Ari Arapostathis.}

\author{Mrinal K. Ghosh}
\address{Department of Mathematics, Indian Institute of Science Bangalore,
Bengaluru, India}
\email{mkg@iisc.ac.in}
%\thanks{Supported in part by  a grant (MATRIX, SERB)
%from the Department of Science and Technology, Government of India}
\author{K.\ Suresh Kumar}
\address{Department of Mathematics, Indian Institute of Technology Bombay,
Powai, Mumbai 400076, India}
\email{suresh@math.iitb.ac.in}
%\thanks{Supported in part by  a grant (MATRIX, SERB)
%from the Department of Science and Technology, Government of India\\
%This paper is dedicated to the memory of Ari Arapostathis.}

\author{Chandan Pal}
\address{Department of Mathematics, Indian Institute of Technology Guwahati,  Guwahati-781039}
\email{cpal@iitg.ac.in}
\author{Somnath Pradhan}
\address{Department of Mathematics and Statistics,  Queen’s University, Kingston, ON, Canada}
\email{sp165@queensu.ca}
%\email{suresh@math.iitb.ac.in}
%\thanks{Supported in part by  a grant (MATRIX, SERB)
%from the Department of Science and Technology, Government of India}

\begin{abstract}
We study  nonzero-sum stochastic differential games with risk-sensitive ergodic cost criterion. Under certain conditions, using multi parameter  eigenvalue approach, we establish the existence of a Nash equilibrium in the space of stationary Markov strategies. We achieve our results by studying the relevant systems of coupled HJB equations. Exploiting the stochastic representation of the principal eigenfunctions we completely characterize Nash equilibrium points in the space of stationary Markov strategies.
\end{abstract}

\maketitle

%\textit{MSC 2010 subject classifications:} Primary  60J60, Secondary 60F10, 93E20
%\medskip


\section{Introduction} We study non zero-sum risk-sensitive stochastic differential games in a multi parameter eigenvalue problem framework. In the literature of stochastic differential games, one usually considers the expectation of the integral of costs (\cite{Borkar-Ghosh}, \cite{ED}, \cite{Va} etc). This is the so called risk-neutral situation where
 the players (i.e., the decision makers or controllers) ignore the risk. If the players are risk-sensitive (i.e., risk-averse or risk-seeking), then one of the most appropriate cost criteria is the expectation of the exponential of the integral of costs as it leads to certainty equivalence \cite{No}. Since the cost criterion is the expectation of the exponential of the integral costs, it is multiplicative as opposed to the additive nature of the cost criterion in the expectation of the integral costs case. Due to this, the analysis of the risk-sensitive case is significantly different from its risk-neutral counterpart. To our knowledge, the risk-sensitive criterion was first introduced by Bellman \cite{RB}; see \cite{PW} and the references therein.
Though this criterion has been studied extensively for stochastic optimal control problems \cite{arapostathis_biswas_saha}, \cite{arapostathis_biswas}, \cite{AAABSP21}, \cite{AAAB2020A}, \cite{BN},\cite{BFN},  \cite{Anup}, \cite{AnupBorkarSuresh}, \cite{FH}, \cite{FM}, \cite{J}, \cite{MR}, \cite{Na}, \cite{SP21}, \cite{Ru}, the corresponding literature in the context of stochastic differential games is rather limited. Some exceptions are \cite{Ba}, \cite{BG}, \cite{Biswas_Saha}, \cite{EH}. Basar \cite{Ba} proves the existence of a  Nash equilibrium for finite horizon nonzero-sum risk sensitive games.  El-Karoui and Hamadene \cite{EH} study risk-sensitive control, zero-sum and nonzero-sum game problems. They prove the existence of an optimal control, a saddle-point and a Nash equilibrium point for relevant cases. In \cite{EH}, authors use Pontrayagin's minimum principle to characterize the optimality condition and the adjoint problem leads to some special  backward stochastic
differential equations. Basu and Ghosh \cite{BG}  study infinite horizon risk-sensitive zero-sum stochastic differential games and establish the
existence of saddle points which are mini-max selectors of the associated Hamilton-Jacobi-Isaacs (HJI) equation. In a recent work Biswas and Saha \cite{Biswas_Saha} consider risk-sensitive zero-sum stochastic differential games for controlled diffusion process in $\mathbb{R}^{d}$. Under fairly general conditions on the  drift and  the diffusion coefficients (e.g.,the  coefficients are locally Lipschitz continuous and have some global growth condition), they  study the ergodic cost criterion. They  completely characterize saddle point equilibria in the space of stationary Markov strategies, under the assumption that running cost function satisfies either small cost condition or dominated by some inf-compact function.

In the framework of reflecting diffusions Ghosh and Pradhan \cite{MKGSP20A} (in bounded domain), \cite{MKGSP22A} (in orthant) have studied similar nonzero-sum game problem for risk-sensitive ergodic cost criterion. Using principal eigenvalue approach, under the assumption that drift term, diffusion matrix and running cost functions are uniformly bounded, they have completely characterized all possible Nash equilibria in the space of stationary Markov strategies.

In this paper we address the existence of Nash equilibria for stochastic differential games where the state of the system is governed by a controlled diffusion processes in the whole space $\RR^d$\,. We consider the risk-sensitive ergodic cost evaluation criterion. We analyze this game problem by analyzing the corresponding system of coupled HJB equation, which is a system of coupled semi-linear elliptic pdes. Under certain conditions, using principal eigenvalue approach we establish the existence of a Nash equilibrium in the space of stationary Markov strategies. In order to establish the existence of principal eigenpair of the associated coupled system of Hamilton-Jacobi-Bellman (HJB) equation, we first study the corresponding Dirichlet eigenvalue problem on smooth bounded domains in $\RR^d$. Applying a version of non-linear Krein–Rutman theorem we show that principal eigenpair exists for Dirichlet eigenvalue problem. Then increasing these domains to $\RR^d$ and employing Fan's fixed  point theorem \cite{Fan} we establish the existence of principal eigenpair to the associated coupled system of HJB equation in the whole space $\RR^d$, which lead to the existence of a Nash equilibrium. Furthermore, exploiting the stochastic representation of the principal eigenfunctions we completely characterize all possible Nash equilibria in the space of stationary Markov strategies. Thus, the main results of this article can be roughly described as follows.
\begin{itemize}
\item[•]\emph{Existence and uniqueness of solution to the coupled HJB equation:} Using Principal eigenvalue approach, we establish the existence and uniqueness of solution to the associated coupled HJB equation in an appropriate function space.
\item[•]\emph{Characterization of Nash equilibrium:} Using Fan's fixed point theorem we first establish the existence of Nash equilibrium in the space of stationary Markov startegies. Then utilizing the stochastic representation of the principal eigenfunctions we completely characterize all possible Nash equilibria in the space of stationary Markov strategies\,.
\end{itemize}
The rest of this paper is organized as follows. Section 2 deals with the problem description. In Section 3 we discuss the principal eigenvlue problem for controlled diffusion operators on smooth bounded domains. Section 4 is devoted to study the eigenvlaue problem for controlled diffusion operator in whole space $\RR^d$\,. The complete characterization of Nash equilibrium in the space of stationary Markov strategies is presented in Section 5\,.
%Section 5 contains some concluding remarks.

\section{Problem Description}
For the sake of notational simplicity we treat two player case.
Let $U_i, i =1,2$ be compact metric spaces and $V_i = {\mathcal P}(U_i)$,  the space of probability  measures on the compact metric space $U_i$ with the topology of weak convergence.
Let
$b : \mathbb{R}^d \times V_1 \times V_2 \to \mathbb{R}^d, \, r_i : \mathbb{R}^d \times V_1 \times V_2 \to \mathbb{R}_{+}$
and  $\sigma : \mathbb{R}^d \to \mathbb{R}^{d \times d}$ be functions such that there exists $\Bar b : \mathbb{R}^d \times U_1 \times U_2 \to \mathbb{R}^d, \, \Bar r_i : \mathbb{R}^d \times U_1 \times U_2 \to \mathbb{R}_{+}$
satisfying
\begin{eqnarray*}
b(x, v_1, v_2)  &  =  &  \iint \Bar b(x, u_1, u_2) v_1(du_1) v_2(du_2),  \\
r_i(x, v_1, v_2)  &  =  &  \iint \Bar r_i(x, u_1, u_2) v_1(du_1) v_2(du_2),  \  i = 1,2,
\end{eqnarray*}
where $\bar{b}, \bar{r}_i, \sigma$ are given functions. 
We consider a nonzero-sum stochastic differential game whose state is evolving according to a controlled diﬀusion process given by the solution of the following stochastic differential equation (s.d.e.)
\begin{equation}\label{statedynamics}
d X(t) =  b(X(t), v_1(t), v_2(t)) dt + \sigma (X(t)) dW(t),
\end{equation}
where $W(\cdot)$ is an $\mathbb{R}^d$-valued standard Wiener process, $v_{i}(\cdot)$ is a $V_{i}$-valued process which is a non-anticipative functional of the state process $X(\cdot),$ i.e., $v_{i}(t) = f_{i}(t, X([0,t]))$ where $X([0,t])(s) = X(s\wedge t)$ for all $s\in [0,\infty)$ and $f_i:[0,\infty)\times C([0,\infty); \mathbb{R}^d)\to V_{i}$. Such a strategy is called an admissible strategy. For $i=1,2,$ $\mathcal{A}_{i}$ denotes the space of all admissible strategies of Player $i.$ In order to ensure the existence of a solution to the equation \ref{statedynamics}, we impose following conditions on the drift term $\bar{b}$, the  dispersion  matrix $\sigma$\,, and the running cost functions $\bar{r}_i, \: i = 1, 2.$
\begin{assumption}\label{A1}
\begin{itemize}
\item[(i)]
\emph{Local Lipschitz continuity:\/}
The function $\sigma\,=\,\bigl[\sigma^{ij}\bigr]\colon\RR^{d}\to\RR^{d\times d}$, $\bar{b}\colon\Rd\times U_1 \times U_2\to\Rd$ and $\bar{r}_i\colon\Rd\times U_1 \times U_2\to\RR_{+}$ are locally Lipschitz continuous in $x$ (uniformly with respect to the rest), i.e., for each $R \geq 0$, there exists a constant $C_{R}>0$ depending on $R>0$, such that
\begin{equation*}
\abs{\bar{b}(x,u_1, u_2) - \bar{b}(y, u_1, u_2)}^2 + \abs{\bar{r}_i(x,u_1, u_2) - \bar{r}_i(y, u_1, u_2)}^2 + \norm{\sigma(x) - \sigma(y)}^2 \,\le\, C_{R}\,\abs{x-y}^2
\end{equation*}
for all $x,y\in B_R$, $i=1,2$ and $(u_1, u_2)\in U_1\times U_2$, where $\norm{\sigma}\df\sqrt{\trace(\sigma\sigma\transp)}$\,. Also, we  assume that $b, r_i$ are  jointly continuous in $(x, u_1, u_2)$ for $i=1,2$.
\medskip
\item[(ii)]
\emph{Affine growth condition:\/} $\bar{b}$ and $\sigma$ satisfy a global growth condition of the form
\begin{equation*}
\sup_{u_1\in U_1, u_2\in U_2}\, \langle \bar{b}(x, u_1, u_2),x\rangle^{+} + \norm{\sigma(x)}^{2} \,\le\,C_0 \bigl(1 + \abs{x}^{2}\bigr) \qquad \forall\, x\in\RR^{d},
\end{equation*}
for some constant $C_0>0$.

\medskip
\item[(iii)]
\emph{Nondegeneracy:\/} For each $R>0$, it holds that
\begin{equation*}
\sum_{i,j=1}^{d} a^{ij}(x)z_{i}z_{j}
\,\ge\,C^{-1}_{R} \abs{z}^{2} \qquad\forall\, x\in B_{R}\,,
\end{equation*}
and for all $z=(z_{1},\dotsc,z_{d})\transp\in\RR^{d}$,
where $a\df \frac{1}{2}\sigma \sigma\transp$.
\end{itemize}
\end{assumption}
It is well known that, under Assumption~\ref{A1}, for any $(v_1, v_2)\in \mathcal{A}_1\times \mathcal{A}_2$ and initial condition $X(0) =x$, the s.d.e. (\ref{statedynamics}) admits a unique weak
solution which is a strong Markov process (see \cite[Theorem 2.2.11, p.42]{AriBorkarGhosh}). For the stochastic differential game, the controlled diffusion given by (\ref{statedynamics}) has the following interpretation: The $i$th player controls the state dynamics, i.e., the controlled diffusion given above, through the choice of her/his strategy $v_{i}$. The function $\bar{r}_i$ represents the running cost function of Player $ $i. If the strategy $v_{i}$ has the form $v_{i}(t) = \Bar{v}_{i}(t, X(t)), t \geq 0$ for some $\bar{v}_i : [0, \ \infty) \times \mathbb{R}^d \to V_i$, then $v_{i}$ or by an abuse of notation $\Bar{v}_i$ is called a Markov strategy for Player $i$. Let ${\mathcal M}_i \ = \
\{ v_i : [0, \ \infty) \times \mathbb{R}^d \to V_i \; | \; v_i \ {\rm is\ measurable}\}$ be the set of all Markov strategies for Player $i$. Under a pair of Markov strategies the s.d.e. (\ref{statedynamics}) admits a unique strong solution which is a strong Markov process (see \cite[Theorem 2.2.12, p.45]{AriBorkarGhosh}). If $v_i$ doesn't have explicit dependence on $t$, i.e., $\bar{v}_i(t, x) = \bar{v}_i(x), \ x \in \mathbb{R}^d, \ t \geq 0$, it is said to be a stationary Markov strategy for Player $i$. The set of all stationary Markov strategies for Player $i$ is denoted by ${\mathcal S}_i, \ i =1,2$. We topologize ${\mathcal S}_i, \ i=1,2$, using a metrizable weak* topology on $L^{\infty}(\mathbb{R}^d ; {\mathcal M}_s(U_i))$, where ${\mathcal M}_{s}(U_i)$ denotes the space of all signed measures on $U_i$ with weak* topology. Since ${\mathcal S}_i$ is a  subset of the unit ball of $L^{\infty}(\mathbb{R}^d ; {\mathcal M}_s(U_i))$, it is compact under the above weak* topology. One also has the following characterization of the topology given by the following convergence criterion:\\
For $i =1,2$, $v^n_i \to v_i$ in ${\mathcal S}_i$ as $n \to \infty$  if and only if
\begin{equation}\label{convergencecriterion}
\lim_{n \to \infty} \int_{\mathbb{R}^d} f(x) \int_{U_i} g(x, u_i) v^n_i(x)(du_i) dx \ = \  \int_{\mathbb{R}^d} f(x) \int_{U_i} g(x, u_i) v_i(x)(du_i) dx ,
\end{equation}
for all $f \in L^1(\mathbb{R}^d) \cap L^2(\mathbb{R}^d), \ g \in C_b(\mathbb{R}^d \times U_i)$; see \cite[p.57]{AriBorkarGhosh},  for details.

For $ v_i \in V_i, i =1,2$,  let 
${\mathcal L}^{v_1, v_2} : C^2(\mathbb{R}^d)   \to C(\mathbb{R}^{d}), $   be given by
\begin{equation}\label{controlleddiffusiononwholespace}
{\mathcal L}^{v_1, v_2} f (x)  \  =  \ a_{ij} (x) \frac{\partial^2 f(x)}{\partial x_i \partial x_j} + b_i (x, v_1, v_2) \frac{\partial f(x)}{\partial x_i},  \ f \in C^2(\mathbb{R}^d),
\end{equation}
where Einstein summation convention is used.
Further, let
 \begin{eqnarray}\label{nonlinearoperators}
 {\mathcal G}^{v_2}_1 f & = & \inf_{v_1 \in V_1} [ {\mathcal L}_1^{v_1, v_2} f
 + r_1 (x, v_1, v_2(x)) f ], \ v_2 \in {\mathcal S}_2, \\ \nonumber
  {\mathcal G}^{v_1}_2 f & = & \inf_{v_2 \in V_2} [ {\mathcal L}_2^{v_1, v_2} f
 + r_2 (x, v_1(x), v_2) f] , \ v_1 \in {\mathcal S}_1 ,  f \in C^2(\mathbb{R}^d),
 \end{eqnarray}
 where for $f \in C^2(\mathbb{R}^d)$,
 \[
 {\mathcal L}_1^{v_1, v_2} f (x)  \  =  \  {\mathcal L}^{v_1, v_2(x)} f (x) \,\,\,\forall\,\,v_1\in V_1, \,\,v_2 \in {\mathcal S}_2
\]
and
\[
{\mathcal L}_2^{v_1, v_2} f (x)  \  =  \  {\mathcal L}^{v_1(x), v_2} f (x) \,\,\,\forall \,\,v_1 \in {\mathcal S}_1\,, \,\,v_2\in V_2\,.
\] For $(v_1, v_2) \in {\mathcal S}_1\times {\mathcal S}_2$, it is easy to see that $${\mathcal L}_1^{v_1, v_2} f(x) = {\mathcal L}_2^{v_1, v_2}f(x) = {\mathcal L}^{v_1, v_2}f(x) =  a_{ij} (x) \frac{\partial^2 f(x)}{\partial x_i \partial x_j} + b_i (x, v_1(x), v_2(x)) \frac{\partial f(x)}{\partial x_i}\,.$$ Our analysis is based on solving the eigenvalue problem associated with the above given operators.

\subsection{Ergodic Cost Criterion}
Given the running cost functions $r_i : \mathbb{R}^d \times V_1 \times V_2  \to \mathbb{R}_{+}, i =1,2$, for any $(v_1, v_2)\in \mathcal{A}_{1}\times \mathcal{A}_{2}$,  the associated risk-sensitive ergodic cost of Player $i$ is defined by
\begin{equation}\label{riskcost}
\rho_i (x, v_1, v_2) \  =  \  \limsup_{T \to \infty} \frac{1}{T} \log \mathbb{E}^{v_1, v_2}_x
\Big[ e^{\int^T_0 r_i (X(t), v_1(t), v_2(t)) dt } \Big],  i =1,2.
\end{equation}
The definition of a  Nash equilibrium is standard, i.e.,  $(v^*_1, v^*_2) \in {\mathcal A}_1 \times
{\mathcal A}_2 $ is a Nash equilbribrium among the class of admissible strategies if
\begin{eqnarray}\label{Nashequilibrium}
\rho_1 (x, v^*_1, v^*_2) & \leq & \rho_1 (x, v_1, v^*_2), \ {\rm for\ all} \
v_1 \in {\mathcal A}_1 , \\ \nonumber
\rho_2 (x, v^*_1, v^*_2) & \leq & \rho_2 (x, v^*_1, v_2), \ {\rm for\ all} \
v_2 \in {\mathcal A}_2, \ {\rm for\ all} \ x \in \mathbb{R}^d.
\end{eqnarray}
We assume that our running cost functions $r_i$\,, $i= 1,2$ satisfy Assumption~\ref{A1}(i)\,.
Now for each $(v_1 , v_2) \in {\mathcal A}_1 \times {\mathcal A}_2$, define
\begin{eqnarray}\label{optimalresponses}
 \lambda_1 (x, v_2) & = & \inf_{v'_1 \in {\mathcal A}_1} \rho_1 (x, v'_1, v_2) ,
 \ \lambda_1 (v_2) \  =  \  \inf_{x \in \mathbb{R}^d} \lambda_1 (x, v_2),\\ \nonumber
  \Lambda_1 (x, v_2) & = & \inf_{v'_1 \in {\mathcal S}_1} \rho_1 (x, v'_1, v_2) ,
 \ \Lambda_1 (v_2) \  =  \  \inf_{x \in \mathbb{R}^d} \Lambda_1 (x, v_2),\\ \nonumber
\lambda_2 (x, v_1) & = & \inf_{v'_2 \in {\mathcal A}_2} \rho_2 (x, v_1, v'_2) ,
 \ \lambda_2 (v_1) \  =  \  \inf_{x \in \mathbb{R}^d} \lambda_2 (x, v_1),\\ \nonumber
  \Lambda_2 (x, v_1) & = & \inf_{v'_2 \in {\mathcal S}_2} \rho_1 (x, v_1, v'_2) ,
 \ \Lambda_2 (v_1) \  =  \  \inf_{x \in \mathbb{R}^d} \Lambda_2 (x, v_1).
 \end{eqnarray}
%%%%%%%%%%%%%%%%%%%%%%%%%%%%%%%%%%%%%%%%%%%%%%%%%%%%%%%%%%%%%%%%%%%%%%%%%%% 
Now we outline our programme for establishing the existence of a Nash equilibrium. We analyze our game problem by analyzing the corresponding system of coupled Hamilton-Jacobi-Bellman (HJB) equations. Suppose that one of the players, say Player 2 announces his strategy $v_2 \in \mathcal{S}_2$ in advance, then Player 1 tries to minimize associated cost $\rho_1(x, v_1, v_2)$ (see, eq. (\ref{riskcost})) over all $v_1 \in \mathcal{A}_1$, which is a (stochastic) optimal control problem for Player 1. Such an optimal control problem has been studied in \cite{arapostathis_biswas_saha}, \cite{Anup}, \cite{AnupBorkarSuresh} and it is shown that one can characterize the optimal value and optimal controls by analyzing the corresponding HJB equation given by
\begin{eqnarray}\label{EOutA}
\lambda_1  \psi_1(x) = \mathcal{G}_1^{v_2}\psi_ 1(x) \quad\text{with}\quad\psi_1(0) = 1\,.
\end{eqnarray} It is well known that (see \cite{arapostathis_biswas_saha}) the principal eigenvalue of the HJB equation is the optimal value and any minimizing selector (\ref{EOutA}), i.e., any $v_1^*\in\mathcal{S}_1$ which satisfies $$\mathcal{G}_{1}^{v_2}\psi_1 = \mathcal{L}_{1}^{v_1^*, v_2}\psi_1 + r_1(x, v_1^*(x), v_2(x))\psi_1$$ is an optimal control for Player 1\,. In particular, $v_1^* \in \mathcal{S}_1$ is an optimal response for Player $1$ corresponding the announced strategy $v_2$ of Player $2$\,. Note that $v_1^*$ depends on $v_2$ and the map
$$ v_2 \,(\in \mathcal{S}_2) \to {\rm the\; optimal \; responses\; of \; Player 1} $$ may be multi-valued. Analogous result holds for Player $2$ if Player $1$ announces his strategy $v_1 \in \mathcal{S}_1$ in advance. From the above discussion, it is easy to see that for any given pair of strategies $(v_1,v_2) \in \mathcal{S}_1 \times \mathcal{S}_2$, one can construct a set of pairs of optimal responses $\{(v_1^*,v_2^*) \in \mathcal{S}_1 \times \mathcal{S}_2\}$ from their corresponding HJB equations. Clearly any fixed point of this multi-valued map is a Nash equilibrium. The above discussion leads to the following programme for finding a pair of Nash equilibrium strategies for ergodic cost criterion. Suppose that there exist a pair of stationary strategies $(v_1^*,v_2^*) \in \mathcal{S}_1 \times \mathcal{S}_2$, a pair of scalars $(\lambda_1, \lambda_2)$ and a pair of functions $(\psi_1,\psi_2)$ in an appropriate function space satisfying the following coupled HJB equations
\interdisplaylinepenalty=0
\begin{align*}
\lambda_1 \psi_1 &= \mathcal{G}_{1}^{v_2^*}\psi_1 = \mathcal{L}_{1}^{v_1^*, v_2^*}\psi_1 + r_1(x, v_1^*(x), v_2^*(x))\psi_1 \\
\lambda_2 \psi_2 &= \mathcal{G}_{2}^{v_1^*}\psi_2 = \mathcal{L}_{2}^{v_1^*, v_2^*}\psi_2 + r_2(x, v_1^*(x), v_2^*(x))\psi_2\,,
\end{align*}
then $(v_1^*,v_2^*)$ will be a pair of Nash equilibrium. 
The above discussion leads us to study the principal eigenvalues associated with the above coupled equations in the subsequent sections.

\section{Dirchlet eigenvalue problem for controlled diffusion operators}
In this section, we  discuss the principal eigenvalue problem
associated with the nonlinear operators ${\mathcal G}^{v_j}_i$ on smooth bounded domains $D\subset \RR^d$\,. The generalized principal eigenvalue of the semi-linear operator ${\mathcal G}^{v_j}_i$ with Dirichlet boundary condition on $D$ is defined by
\begin{equation}\label{dirichletprincipaleigen}
 \lambda^+_{i}(v_j, D) \  =  \ \inf \{ \lambda \in \mathbb{R} \mid \ {\rm for\ some} \  \varphi \in
 W^{2, p}_{loc}(D) \cap C(\Bar{D}), \varphi > 0, {\mathcal G}^{v_j}_i \varphi \leq \lambda \varphi
 \ {\rm in} \ D \},\,\, i, j = 1, 2.
 \end{equation} Now we  prove the existence of the principal eigenvalues of a certain parametric family of semi-linear elliptic pdes.
\begin{theorem}\label{Dirichleteigenvalue} Let $v_j \in  {\mathcal S}_j$
and $D$ be a bounded smooth domain in $\mathbb{R}^d$. Then there exists (unique upto a scalar multiplication)
$\psi_D \in W^{2,p} (D) \cap C (\Bar{D}),  \psi_D > 0$ such that
\begin{eqnarray}\label{Dirichleteigenvalueproblem}
{\mathcal G}^{v_j}_i  \psi_D  &  =  &   \lambda^+_i(v_j, D)  \psi_D,  \\ \nonumber
\psi_D & = & 0 \  {\rm on} \  \partial D\,,\,\quad i, j = 1,2\,.
\end{eqnarray}
\end{theorem}

\begin{proof} We take $i=1, j=2$. Suppose $r_1 \leq 0$ (this will be dropped shortly). For $\phi\in C^{1}_{0}(D) (:= C_{0}(\bar{D})\cap C^{1}(D))$, $f \in L^p(D)$,  
 let $$\Gamma_1(\phi, f)(x) = - \inf_{v_1\in V_1}\{b_i (x, v_1, v_2(x)) \frac{\partial \phi(x)}{\partial x_i} + r_1 (x, v_1, v_2(x)) \phi(x)\} + f(x),$$ and consider
\begin{equation}\label{ET3.2A1}
a_{ij} (x) \frac{\partial^2 \hat{\phi}(x)}{\partial x_i \partial x_j} = \Gamma_1(\phi, f)(x)\,,\quad\text{with}\,\,\, \hat{\phi} = 0 \,\,\,\text{on}\quad \partial{D}\,.
\end{equation}
Then by \cite[Theorem~9.15, p. 241]{GilbargTrudinger}, \cite[Theorem~9.14, p.240]{GilbargTrudinger}, there exists a unique solution $\hat{\phi}\in W^{2,p}(D) \cap C(\Bar{D}),$ $p > d$, satisfying
\begin{equation}\label{ET3.2A}
\|\hat{\phi}\|_{W^{2,p}(D)} \leq \kappa_1 (\|\hat{\phi}\|_{\infty} + \|\Gamma_1(\phi, f)\|_{L^{p}(D)})\,,
\end{equation} for some positive constant $\kappa_1 = \kappa_1(p, D)$ which is independent of $\hat{\phi}$, $\phi$, $f$. From \cite[Theorem~9.1, p. 220]{GilbargTrudinger}, we deduce that
\begin{equation*}
\|\hat{\phi}\|_{\infty} \leq \kappa_2 \|\Gamma_1(\phi, f)\|_{L^{d}(D)},
\end{equation*} for some constant $\kappa_2 > 0$. Hence, from (\ref{ET3.2A}), we obtain
\begin{equation}\label{ET3.2B}
\|\hat{\phi}\|_{W^{2,p}(D)} \leq \kappa_3 \|\Gamma_1(\phi, f)\|_{L^{p}(D)}
\end{equation} for some positive constant $\kappa_3$\,. Now consider an operator $\mathfrak{T}$ mapping $\phi\in C_0^1(D)$ to the corresponding solution $\hat{\phi}$ of (\ref{ET3.2A1}), i.e., $\mathfrak{T}(\phi) = \hat{\phi}$\,. Since the embedding $W^{2,p}(D)\hookrightarrow C^{1,\alpha}(D)$ for $p > d$ and $\alpha\in (0, 1-\frac{d}{p})$ is compact, the operator $\mathfrak{T}$ is compact and continuous. Now we want to show that the following space of functions
\begin{equation*}
\{\phi\in C_0^1(D) : \phi = \nu\mathfrak{T}(\phi)\quad \text{for some}\quad \nu\in [0,1]\},
\end{equation*}is bounded in $C_0^1(D)$. Suppose that there exists a sequence $(\phi_n, \nu_n)$ with $\|\phi_n\|_{C_0^1(D)}\to\infty$ and $\nu_n\to\nu\in[0,1]$ as $n\to \infty$\,. Scaling $\phi_n$ appropriately we assume that $\|\phi_n\|_{C_0^1(D)} = 1$\,. Hence, in view of the estimate (\ref{ET3.2B}), extracting a suitable subsequence, there exists a nontrivial $\tilde{\phi}$ satisfying
\begin{equation*}
a_{ij} (x) \frac{\partial^2 \tilde{\phi}(x)}{\partial x_i \partial x_j} = - \nu \inf_{v_1\in V_1}\{b_i (x, v_1, v_2(x)) \frac{\partial \tilde{\phi}(x)}{\partial x_i} + r_1 (x, v_1, v_2(x)) \tilde{\phi}(x)\}\,,
\end{equation*} with $\tilde{\phi} = 0$ on $\partial{D}$. This is a contradiction to the strong maximum principle \cite[Theorem~9.6, p. 225]{GilbargTrudinger}. This implies that the above space is bounded. Hence, by the Leray-Schauder fixed point theorem \cite[Theorem~11.3, p. 280]{GilbargTrudinger}, it follows that $\mathfrak{T}$ admits a fixed point $\varphi \in W^{2,p} (D) \cap C (\Bar{D})\,,$ i.e., we have
$${\mathcal G}^{v_2}_1  \varphi(x) = f(x),\quad\text{with}\quad \varphi = 0 \quad\  {\rm on}\quad\  \partial D.$$
Also, by the strong maximum principle \cite[Theorem~9.6]{GilbargTrudinger} it is clear that $\varphi$ satisfying the equation is unique.

Let $\mathfrak{X} = C_{0}(D)$ and $\mathfrak{C}$  the cone of non-negative functions in $\mathfrak{X}$. Now define an operator $\hat{\mathfrak{T}}$ which maps $f\in\mathfrak{X}$ to corresponding solution $\varphi\in W^{2,p} (D) \cap C (\Bar{D})$ satisfying
$${\mathcal G}^{v_2}_1  \varphi(x) = - f(x),\quad\text{with}\quad \varphi = 0 \quad\  {\rm on}\quad\  \partial D.$$ From the above discussion it is easy to see that the operator $\hat{\mathfrak{T}}$ is well defined. Thus, combining \cite[Theorem~9.1]{GilbargTrudinger} and \cite[Theorem~9.14]{GilbargTrudinger}, we deduce that
\begin{equation}\label{ET3.2C}
\|\varphi\|_{W^{2,p}(D)} \leq \kappa_1 \sup_{D}|\varphi|\,,
\end{equation}
for some positive constant $\kappa_1$\,. From (\ref{ET3.2C}), it is clear that $\hat{\mathfrak{T}}$ is compact and continuous. Also, from the definition one can see that $\hat{\mathfrak{T}}$ is $1$-homogeneous (i.e., $\hat{\mathfrak{T}}(\tilde{\lambda} f) = \tilde{\lambda} \hat{\mathfrak{T}}(f)$ for all $\tilde{\lambda} \geq 0$). Suppose $\hat{\mathfrak{T}}(f_k) = \varphi_k$, $k=1,2$\,, with $f_1 \leq f_2$. Thus, we have ${\mathcal G}^{v_2}_1  \varphi_1(x) \geq {\mathcal G}^{v_2}_1  \varphi_2(x)$\,. Since ${\mathcal G}^{v_2}_1$ is concave, it follows that ${\mathcal G}^{v_2}_1  (\varphi_2 - \varphi_1)(x) \leq 0$\,. Hence, applying \cite[Theorem~3.1]{QuaasSirakov} we obtain $\varphi_2 \geq\varphi_1$ and if $f_1 < f_2$ (i.e., $f_1 \leq f_2$ and $f_1 \neq f_2$) then we have $\varphi_2 > \varphi_1$ (see \cite[Lemma~3.1]{QuaasSirakov}). This implies that $\hat{\mathfrak{T}}$ is order preserving. Let $\Tilde{\phi}\in\mathfrak{C}$ be nontrivial nonnegative function with compact support, hence from the above discussion we deduce that $\hat{\mathfrak{T}}(\Tilde{\phi}) > 0$. Thus, one can choose $\kappa_2 > 0$ such that $\kappa_2\hat{\mathfrak{T}}(\Tilde{\phi}) - \Tilde{\phi} > 0$ in $D$\,. Therefore, by Krein-Rutman theorem (see Theorem~\ref{TAA.1}), we conclude that there exists $(\hat{\lambda},  \psi_D) \in \mathbb{R}_{+} \times W^{2,p} (D) \cap C(\Bar{D})$ with $\psi_D >0$ satisfying
\begin{equation}\label{ET3.2D}
{\mathcal G}^{v_2}_1 \psi_D = \hat{\lambda}\psi_D \quad\text{in}\quad D, \quad\text{and}\quad \psi_D = 0\quad\text{on}\quad\partial{D}\,.
\end{equation}Where $\psi_D$ is unique upto scalar multiplication. Now,  $r_1 \geq 0$ (which is the case by our assumption), since $r_1$ is bounded in $\bar{D}$ replacing $r_1$ by $(r_1-\|r_1\|_{\infty, D})$\,, following the above arguments there exists $(\lambda_D,  \psi_D) \in \mathbb{R} \times W^{2,p} (D) \cap C(\Bar{D})$ with $\psi_D >0$ satisfying (\ref{ET3.2D}).

Next, we show that
\[
\lambda_D \  =  \  \lambda^{+}_1 (v_2, D).
\]
Clearly,
\begin{equation}\label{onesideinequalitythm3.2}
\lambda_D \geq \lambda^{+}_1 (v_2, D).
\end{equation}
Suppose $\lambda^{+}_1 (v_2, D)  < \lambda_D$. Then for each $\varepsilon > 0$, there exists $\varepsilon' \leq \varepsilon $ and $\varphi' \in W^{2, p}(D) \cap C(\Bar{D}), \varphi' > 0$
such that
\begin{equation}\label{eq3thm3.2}
{\mathcal G}^{v_2}_1 \varphi' \leq (\lambda^{+}_1 (v_2, D) + \varepsilon' ) \varphi'.
\end{equation} Choose $\epsilon > 0$ small enough such that $\lambda^{+}_1 (v_2, D) + \varepsilon' < \lambda_D$. Also, we have
\begin{equation}\label{eq4thm3.2}
{\mathcal G}^{v_2}_1 \psi_D - (\lambda^{+}_1 (v_2, D) + \varepsilon' ) \psi_D
\  > \  {\mathcal G}^{v_2}_1 \psi_D - \lambda_D \psi_D  = 0.
\end{equation}
Hence by Theorem \ref{auxillaryresult}, it follows that $\psi_D = t \varphi'$ for some $t > 0$. This gives a contradiction. Therefore we get $\lambda_D = \lambda^{+}_1 (v_2, D)$. This completes the proof.
\end{proof}

\section{Eigenvalue problem for controlled diffusion operators in $\mathbb{R}^d$}
In this section we explore the existence of generalised eigenvalue of the controlled
diffusion operator ${\mathcal G}^{v_j}_i, v_j \in {\mathcal A}_j$ in the whole space $\mathbb{R}^d$ and establish their relations with the risk-sensitive ergodic optimal control problem\,. The generalized principal eigenvalue of ${\mathcal G}^{v_j}_i$ in the whole space is defined by
\begin{equation}\label{principaleigenvalue}
\lambda^+_i(v_j) \ =  \  \inf \{ \lambda \in \mathbb{R} \mid \ {\rm for\ some} \  \varphi \in
 W^{2, d}_{loc} (\mathbb{R}^d) \cap C(\mathbb{R}^d), \varphi > 0,
 {\mathcal G}^{v_j}_i \varphi \leq \lambda \varphi  \ {\rm a.e.} \} .
 \end{equation}

In order to study our game problem we enforce following Foster-Lyapunov condition on the dynamics.	

\begin{assumption}\label{A2}
\begin{itemize}
\item[(i)]\textbf{In bounded cost case:} There exist  ${\mathcal V} \in C^2(\mathbb{R}^d)$ with
$\inf_{\mathbb{R}^d} {\mathcal V}  \geq 1$, constants $\delta , \tilde{\alpha}  > 0$
and a compact set $\tilde{K}$
such that
\begin{equation}\label{EA21}
\sup_{u_i \in \mathbb{U}_i, i =1,2} {\mathcal L}^{u_1, u_2} {\mathcal V} \leq \tilde{\alpha} I_{\tilde{K}} -  \delta   {\mathcal V}.
\end{equation}
and $ \max_{i=1,2}\|r_i \|_\infty < \delta.$

Or,

\item[(ii)]\textbf{In unbounded cost case:} There exist  ${\mathcal V} \in C^2(\mathbb{R}^d)$ with
$\inf_{\mathbb{R}^d} {\mathcal V}  \geq 1$, an inf-compact positive  $\ell \in C(\mathbb{R}^d)$ (i.e., the sublevel sets $\{\ell \leq \kappa \}$ are compact, or empty, in $\mathbb{R}^d$ for each $\kappa\in \mathbb{R}$), a constant $\tilde{\alpha } > 0$ and a compact set $\tilde{K}$ such that
\begin{equation}\label{EA22}
\sup_{u_i \in \mathbb{U}_i, i =1,2} {\mathcal L}^{u_1, u_2} {\mathcal V} \leq \tilde{\alpha}
I_{\tilde{K}} - \ell {\mathcal V},
\end{equation} and for $i = 1,2$
\begin{equation}\label{EA23}
\ell(x) - \sup_{u_i \in \mathbb{U}_i, i =1,2}r_{i}(x, u_1, u_2)\quad \text{is inf-compact}\,.
\end{equation}
\end{itemize}
\end{assumption}
As noted in \cite{arapostathis_biswas_saha}, \cite{AAABSP21}, if $a$ and $b$ are bounded, it might not be possible to find an unbounded function $\ell$ which satisfies (\ref{EA22}). In view of this, we are assuming (\ref{EA21}).

For $i \neq j$, it is easy to see that under Assumption~\ref{A2}(i)
\[
\sup_{v_1 \in {\mathcal A}_1} \sup_{v_2 \in {\mathcal A}_2}
\limsup_{T \to \infty} \frac{1}{T} \log \mathbb{E}^{v_1, v_2}_x
\Big[ e^{\int^T_0 r_i(X(t), v_1(t), v_2(t)) dt } \Big] \leq \|r_i \|_\infty < \infty\,.
\] Also, under Assumption~\ref{A2}(ii), applying It\^{o}-Krylov formula, it follows that
\[
\sup_{v_1 \in {\mathcal A}_1} \sup_{v_2 \in {\mathcal A}_2}
\limsup_{T \to \infty} \frac{1}{T} \log \mathbb{E}^{v_1, v_2}_x
\Big[ e^{\int^T_0 \ell(X(t)) dt } \Big] \leq \frac{\tilde{\alpha}}{\min_{\tilde{K}}\mathcal{V}}\,.
\]
From (\ref{EA23}), it is clear that $\displaystyle{\sup_{u_k \in \mathbb{U}_k, k=1,2} r_i (\cdot, u_1, u_2) \leq \kappa_1 + \ell(\cdot)}$, for some positive constant $\kappa_1$\,.  Therefore, we obtain
\begin{equation}\label{EEst1}
\sup_{v_1 \in {\mathcal A}_1} \sup_{v_2 \in {\mathcal A}_2}
\limsup_{T \to \infty} \frac{1}{T} \log \mathbb{E}^{v_1, v_2}_x
\Big[ e^{\int^T_0 r_i(X(t), v_1(t), v_2(t)) dt } \Big] \leq \kappa_1 + \frac{\tilde{\alpha}}{\min_{\tilde{K}}\mathcal{V}}\,.
\end{equation}
Now we proceed  to prove the existence of the principal eigenpair to certain semi-linear elliptic pdes in the whole space $\RR^d$\,.
\begin{theorem}\label{existenceeigenvalue} Let Assumptions \ref{A1} and \ref{A2} hold. Suppose $v_j\in \mathcal{S}_j$, then there exists a unique
$\psi \in   W^{2, p}_{loc} (\mathbb{R}^d) \cap C(\mathbb{R}^d),
p \geq 2, \psi > 0$ such that
\begin{equation}\label{ET4.1AA}
{\mathcal G}^{v_j}_i \psi \  =  \  \lambda^+_i(v_j) \psi  \quad\text{with}\,\,\,\psi(0) = 1.
\end{equation}
Moreover  $\lambda^+_i(v_j)$ is simple and satisfies
\begin{equation}\label{EOpt1A}
\lambda^+_i (v_j) \  \leq \  \lambda_i(v_j)\,, \quad\text{for}\quad i \neq j,\, i,j = 1,2\,.
\end{equation}
\end{theorem}
\begin{proof} Take $i=1, j=2$. Let $D = B_n, n \geq 1 $, denote the open ball centred at the origin with radius $n$.  From Theorem~\ref{Dirichleteigenvalue}, there exists a (unique) $\psi_n \in W^{2,p} (B_n) \cap C(\Bar{B}_n), \psi_n > 0$ in $B_n$ with $\psi_n(0) = 1$ satisfying
\begin{eqnarray}\label{EPrincipleEigen1A}
{\mathcal G}^{v_2}_1 \psi_n & =  & \lambda_n \psi_n \nonumber\\
\psi_n & = & 0 \ {\rm on} \ \partial B_n,
\end{eqnarray}
where $\lambda_n = \lambda_1^+(v_2, B_n)$.  Choose $v_1 \in {\mathcal A}_1$, since $\psi_n = 0$ on $\partial{B}_n$ applying Ito-Dynkin's formula we obtain 
\begin{align*}
\psi_n(x) &\leq \mathbb{E}^{v_1, v_2}_x\left[e^{\int_{0}^{T}(r_1(X(t), v_1(t), v_2(X(t)))-\lambda_n) dt} \psi_n(X(T)) \Ind_{\{T \leq \tau\}}\right]\\
&\leq \|\psi_n\|_{\infty , B_n}\mathbb{E}^{v_1, v_2}_x\left[e^{\int_{0}^{T}(r_1(X(t), v_1(t), v_2(X(t)))-\lambda_n) dt}\right]\quad\text{for all}\quad (T,x)\in \RR_+\times B_n\,,
\end{align*} where $\tau$ is the first exit time of the process $X(t)$ from $B_n$. Thus, taking logarithm on both sides of the inequality, dividing by $T$ and letting $T\to\infty$, it follows that
\begin{equation}\label{ET4.1A}
\lambda_n \ \leq \ \limsup_{T \to \infty} \frac{1}{T} \log \mathbb{E}^{v_1, v_2}_x \Big[ e^{\int^T_0 r_1(X(t), v_1(t), v_2(X(t))) dt } \Big]  < \infty .
\end{equation}
Since $\lambda_n $ is nondecreasing in $n$ (see, (\ref{dirichletprincipaleigen})), it follows that $\lim_n \lambda_n = \lambda$
exists.

Now using Harnack inequality (see \cite[ Corollary 8.21, p.199]{GilbargTrudinger}) and the interior estimates \cite[Theorem 9.11, p.235]{GilbargTrudinger}, we get for each bounded domain $D$, there exists $n_0$ such that
\begin{equation}\label{ET4.1B}
\sup_{n \geq n_0} \|\psi_n\|_{2, p, D} < \infty.
\end{equation}
Hence, by a standard diagonalization procedure and Banach-Alaoglu theorem, we can extract a subsequence $\{\psi_{n_k}\}$ such that for some $\psi \in W^{2, p}_{loc} (\mathbb{R}^d) \cap C(\mathbb{R}^d), p \geq 2$
\begin{equation}\label{ETC1.3BC}
\begin{cases}
\psi_{n_k}\to & \psi\quad \text{in}\quad W^{2, p}_{loc} (\mathbb{R}^d)\quad\text{(weakly)}\\
\psi_{n_k}\to & \psi\quad \text{in}\quad C^{1, \alpha}(K) \quad\text{(strongly)}\,\,\, \text{for all compact set}\,\, K \subset \Rd \,, 
\end{cases}       
\end{equation} where $0 < \alpha < 1 - \frac{d}{p}$\,. Now multiplying  both sides of (\ref{EPrincipleEigen1A}) by a test function $\varphi\in C_{c}^{\infty}(\RR^d)$,  integrating, and then letting $n\to \infty$, we deduce that $\psi \in W^{2, p}_{loc} (\mathbb{R}^d) \cap C(\mathbb{R}^d), p \geq 2$ satisfies
\begin{equation}\label{ET4.1C}
{\mathcal G}^{v_2}_1 \psi \  =  \  \lambda \psi \  {\rm in} \  \mathbb{R}^d.
\end{equation}
From (\ref{ET4.1A}), it follows that
\[
\lambda \leq \lambda_1 (v_2).
\]
Since for each $n\in\mathbb{N}$ we have $\psi_n > 0$ it clear that $\psi \geq 0$ in $\mathbb{R}^d$ and since $\psi_n(0) = 1$ for all $n$, we have $\psi(0)=1$. Thus, applying Harnack's inequality we deduce that $\psi > 0$ in $\mathbb{R}^{d}$.

Next from the definition of the  generalized principal eigenvalue, it is immediate that
\begin{equation}\label{ET4.1D}
\lambda \geq \lambda^+_{1}(v_2).
\end{equation}
Also from the definition of the generalized principle eigenvalue (see eq. (\ref{dirichletprincipaleigen})), it follows that
\begin{equation}\label{eq5thm4.1}
\lambda_n = \lambda^+_1(v_2, B_n) \leq \lambda^+_1(v_2).
\end{equation}
Thus, combining (\ref{ET4.1D}) and (\ref{eq5thm4.1}) we get
\[
\lambda = \lambda^+_1(v_2).
\]
Next we show that any eigenvalue of ${\mathcal G}^{v_2}_1$ corresponding to a positive eigenfunction in the class $W^{2, p}_{loc}(\mathbb{R}^d) \cap C(\mathbb{R}^d)$ is simple. This, in particular,  would  impliy the simplicity of the generalized principal eigenvalue $\lambda^+_1(v_2)$.

Let $\psi_k \in W^{2, p}_{loc}(\mathbb{R}^d) \cap C(\mathbb{R}^d), k=1,2$ be positive eigenfunctions corresponding to an eigenvalue $\lambda$ (in particular, we are interested in $\lambda = \lambda^+_1(v_2)$) satisfying $\psi_k (0) =1$. Let $t_0 > 0 $ be such that $\psi_1- t_0 \psi_2 \geq 0 $ in $\Bar B_R$.

Let $v_1$ be a minimizing selector of ${\mathcal G}^{v_2}_1 \psi_1 $.
Thus
\begin{eqnarray*}
\mathcal{L}_1^{v_1, v_2} \psi_1 + r_1(x, v_1(x), v_2(x)) \psi_1 &  = & {\mathcal G}^{v_2}_1\psi_1  = \lambda \psi_1 \\
 \mathcal{L}_1^{v_1, v_2} \psi_2  +  r_1(x, v_1(x), v_2(x)) \psi_2 &  \geq  & {\mathcal G}^{v_2}_1\psi_2 = \lambda \psi_2\,.
\end{eqnarray*}
This gives us the following inequality
\[
\mathcal{L}_1^{v_1, v_2} ( \psi_1 -t_0 \psi_2) + r_1(x, v_1(x), v_2(x)) (\psi_1 - t_0 \psi_2)  \  \leq \
\lambda (\psi_1 - t_0 \psi_2).
\]
Since $\psi_1 - t_0 \psi_2 \geq 0$ in $\Bar B_R$, it follows that
\[
\mathcal{L}_1^{v_1, v_2} ( \psi_1 -t_0 \psi_2) - (r_1(x, v_1(x), v_2(x)) -\lambda)^{-} (\psi_1 - t_0 \psi_2)
\leq 0 \  {\rm in} \  B_R.
\]
Hence using the  maximum principle \cite[Theorem~9.6]{GilbargTrudinger}, we have $\psi_1 - t_0 \psi_2 =0 $ in $B_R$ and since $\psi_1 (0) = \psi_2(0) =1$, we get $t_0 =1$ and hence $\psi_1 = \psi_2$ in $B_R$. Since the choice of $R>0$ arbitrary (by choosing large $R>0$), it follows that $\psi_1 = \psi_2$ in $\mathbb{R}^d$. This completes the proof.
\end{proof}
We denote the eigenfunction corresponding to $\lambda^+_i(v_j)$ satisfying $\psi (0) =1$ by $\psi_i(v_j)$. Next theorem proves that the eigenfunction $\psi_i(v_j)$ corresponding to the principal eigenvalue $\lambda^+_i(v_j)$ admits certain stochastic representation. This result  plays  crucial role in obtaining complete characterization of Nash equilibrium in the space of stationary Markov strategies.
\begin{theorem}\label{stochasticrepeigenfunction}  Let Assumptions \ref{A1}, \ref{A2} hold. Then, for $v_j\in \mathcal{S}_j$, the eigenfunction $\psi_i(v_j)$ corresponding to principal eigenvalue $\lambda^+_i(v_j)$ satisfies
\begin{equation}\label{ET4.2AA}
\psi_i(v_j) (x) \  =  \  \mathbb{E}^{v_1, v_2}_x \Big[ e^{\int^{\breve{\tau}_r}_0 (r_i(X(t), v_1(X(t)), v_2(X(t))) - \lambda^+_i(v_j) ) dt } \psi_i(v_j) (X(\breve{\tau}_r) \Big], r > 0\,,
\end{equation}
where $\breve{\tau}_r $ is the hitting time of $X(t)$ to $B_r$ and $v_i\in {\mathcal S}_i$ is a minimizing selector of ${\mathcal G}^{v_j}_i \psi_i(v_j)$,\, $i,j = 1,2$.
\end{theorem}
\begin{proof} Take $i=1, j=2$.
Let $(\hat \lambda_n, \hat \psi_n)$ denote the generalized principal eigenpair
of the Dirichlet eigenvalue problem of ${\mathcal L}_1^{v_1, v_2} + r_1(x, v_1(x), v_2(x))$ in $B_n$ with $\hat \psi_n (0) =1$.
Using the monotonicity of $\hat \lambda_n$ with respect to the running cost, the following representation holds (see, \cite[Lemma~2.10 (i)]{arapostathis_biswas})
\begin{equation}\label{stochasticrepdirichleteigenfn}
\hat \psi_n(x) \  =  \  \mathbb{E}^{v_1, v_2}_x \Big[ e^{\int^{\breve{\tau}_r}_0 (r_1(X(t), v_1(X(t)), v_2(X(t)))
- \hat \lambda_n) dt } \hat \psi_n (X(\breve{\tau}_r)) I \{ \breve{\tau}_r < \tau_n\} \Big],
\end{equation}
where $\tau_n = \tau(B_n),$ the exit time from $B_n$. Also as in the proof of
Theorem \ref{existenceeigenvalue}, it follows that $\hat \lambda_n \uparrow \lambda_1^{+}(v_1, v_2)$, the generalized principal eigenvalue of ${\mathcal L}_1^{v_1, v_2} + r_1(x, v_1(x), v_2(x))$. Again using Harnack's inequality
and the standard approximation argument (as in Theorem \ref{existenceeigenvalue}), it follows that there exists
$\psi \in W^{2, p}_{loc}(\mathbb{R}^d \cap C(\mathbb{R}^d), \psi > 0$ satisfying
\begin{equation}\label{eq1.1thm4.2}
{\mathcal L}_1^{v_1, v_2} \psi  + r_1(x, v_1(x), v_2(x)) \psi \ = \ \lambda_1^+(v_1, v_2) \psi .
\end{equation}
Consider
\begin{align}\label{eq1.2thm4.2}
&\mathbb{E}^{v_1, v_2}_x \Big[ e^{\int^{\breve{\tau}_r}_0 (r_1(X(t), v_1(X(t)), v_2(X(t))) - \hat \lambda_n) dt } \hat \psi_n (X(\breve{\tau}_r)) I \{ \breve{\tau}_r < \tau_n\} \Big] \nonumber  \\
\  \leq &\ \mathbb{E}^{v_1, v_2}_x \Big[ e^{\int^{\breve{\tau}_r}_0 (r_1(X(t), v_1(X(t)), v_2(X(t))) - \hat \lambda_n) dt }  \psi (X(\breve{\tau}_r)) I \{ \breve{\tau}_r < \infty\} \Big]\\ \nonumber
& + \sup_{\partial B_r} |\hat{\psi}_n - \psi |
\mathbb{E}^{v_1, v_2}_x \Big[ e^{\int^{\breve{\tau}_r}_0 (r_1(X(t), v_1(X(t)), v_2(X(t))) - \hat \lambda_n) dt }  I \{ \breve{\tau}_r < \tau_n\} \Big] \,.
\end{align}
Using the  monotone convergence theorem, the first term in the r.h.s. of (\ref{eq1.2thm4.2}) converges  to
\begin{equation*}
\mathbb{E}^{v_1, v_2}_x \Big[ e^{\int^{\breve{\tau}_r}_0 (r_1(X(t), v_1(X(t)), v_2(X(t))) - \lambda_1^+(v_1, v_2)) dt }  \psi (X(\breve{\tau}_r)) \Ind_{\{ \breve{\tau}_r < \infty\}} \Big].
\end{equation*}
The second term
\begin{align*}
 &\sup_{\partial B_r} |\hat{\psi}_n - \psi |
\mathbb{E}^{v_1, v_2}_x \Big[ e^{\int^{\breve{\tau}_r}_0 (r_1(X(t), v_1(X(t)), v_2(X(t))) - \hat \lambda_n) dt }  I \{ \breve{\tau}_r < \tau_n\} \Big]  \\
&\leq \  \frac{ \sup_{\partial B_r} |\hat{\psi}_n - \psi | }{\inf_{\partial B_r} \hat{\psi}_n}
\mathbb{E}^{v_1, v_2}_x \Big[ e^{\int^{\breve{\tau}_r}_0 (r_1(X(t), v_1(X(t)), v_2(X(t)))
- \hat \lambda_n) dt } \hat{\psi}_n (X(\breve{\tau}_r)) I \{ \breve{\tau}_r < \tau_n\} \Big] \\
& =  \   \frac{ \sup_{\partial B_r} |\hat{\psi}_n - \psi | }{\inf_{\partial B_r} \hat{\psi}_n}  \hat{\psi}_n (x) \,\,
\to \  0  \  {\rm as } \  n \to \infty.
\end{align*}
In the above, we have used the fact that $\hat{\psi}_n - \psi \to 0$ uniformly over compact sets and $\inf_{\partial B_r} \hat{\psi}_n > 0$ (by Harnack's inequality)\,. Hence, we get
\begin{equation}\label{eq2.1thm4.2}
\psi (x)  \  \leq  \  \mathbb{E}^{v_1, v_2}_x \Big[ e^{\int^{\breve{\tau}_r}_0 r_1 (X(t), v_1(X(t)), v_2(X(t))) -
\lambda_1^+ (v_1, v_2)) dt } \psi (X(\breve{\tau}_r)) \Big] .
\end{equation}

Since
\begin{equation*}
{\mathcal G}^{v_2}_1 \hat \psi_n \leq {\mathcal L}_1^{v_1, v_2} \hat \psi_n + r_1(x, v_1(x), v_2(x))\hat \psi_n
  = \hat \lambda_n \hat \psi_n\,,
\end{equation*}  it follows that
\begin{equation*}
\lambda^+_1(v_2, B_n)  \leq \hat \lambda_n, n \geq 1.
\end{equation*}
Therefore
\begin{equation*}
\lambda^+_1(v_2) \leq \lambda_1^+(v_1, v_2).
\end{equation*}
Since $v_1\in\mathcal{S}_{1}$ is a minimizing selector of ${\mathcal G}^{v_2}_1 \psi_1(v_2)$, we have

\begin{equation}\label{eq1thm4.2}
{\mathcal L}_1^{v_1, v_2} \psi_1(v_2) + r_1(x, v_1(x), v_2(x)) \psi_1(v_2)
\  =  \ \lambda^+_1(v_2) \psi_1(v_2).
\end{equation}
Using Ito-Krylov formula, for fixed $T > 0$, $x \in B^c_r\cap B_n,  \  r > 0$ and $n$ large enough, we have
\begin{equation*}
\psi_1 (v_2)(x) =  \mathbb{E}^{v_1, v_2}_x \Big[ e^{\int^{\breve{\tau}_r\wedge T\wedge \tau_{n}}_0 r_1 (X(t), v_1(X(t)), v_2(X(t))) - \lambda^+_1 (v_2)) dt } \psi_1(v_2) (X(\breve{\tau}_r\wedge T\wedge \tau_{n})) \Big].
\end{equation*} Lettting $n\to\infty,$ and $T\to\infty$ and using Fatou's lemma, it follows that
\begin{eqnarray}\label{eq2thm4.2}
\psi_1 (v_2)(x) & \geq &
\mathbb{E}^{v_1, v_2}_x \Big[ e^{\int^{\breve{\tau}_r}_0 r_1 (X(t), v_1(X(t)), v_2(X(t))) -
\lambda^+_1 (v_2)) dt } \psi_1(v_2) (X(\breve{\tau}_r)) \Big] \\ \nonumber
\psi_1 (v_2)(x) &   \geq   &
\mathbb{E}^{v_1, v_2}_x \Big[ e^{\int^{\breve{\tau}_r}_0 r_1 (X(t), v_1(X(t)), v_2(X(t))) -
\lambda_1^+ (v_1, v_2)) dt } \psi_1(v_2) (X(\breve{\tau}_r) \Big] .
\end{eqnarray}
Hence  for each $ t > 0$,
\begin{equation}\label{eq3thm4.2}
\psi_1(v_2)(x) - t \psi (x) \  \geq \
 \mathbb{E}^{v_1, v_2}_x \Big[ e^{\int^{\breve{\tau}_r}_0 r_1 (X(t), v_1(t), v_2(t)) -
\lambda_1^+ (v_1, v_2)) dt } (\psi_1(v_2)(X(\breve{\tau}_r)) -  t \psi (X(\breve{\tau}_r)) \Big].
\end{equation}
Thus
\begin{equation*}
\psi_1(v_2)(x) -  t \psi (x)  \geq 0 \  {\rm in} \,\, \Bar{B}_r \quad\text{implies that}\quad \psi_1(v_2)(x) -  t \psi (x)  \geq 0
\ {\rm in} \,\, \mathbb{R}^d.
\end{equation*}
Now choose $t > 0$ such that $\psi_1(v_2)(x) -  t \psi (x)  \geq 0 $ in $\Bar{B}_r$ and attains its minimum value $0$ in $\Bar{B}_r$. Hence $\psi_1(v_2)(x) -  t \psi (x)  \geq 0 $ in $\mathbb{R}^d$ and attains its
minimum in $\mathbb{R}^d$. Now using $\lambda^+_1 (v_2) \leq \lambda^+(v_1, v_2)$,
it is easy to verify that
\[
{\mathcal L}_1^{v_1, v_2} (\psi_1(v_2)-  t \psi ) - (r_1(x, v_1(x), v_2(x)) - \lambda^+_1(v_2) )^{-}
(\psi_1(v_2)-  t \psi )  \leq 0.
\]
Hence using the  strong maximum principle \cite[Theorem~9.6]{GilbargTrudinger}, we get $\psi_1 (v_2) = t \psi$. Since
$\psi_1 (v_2) (0) = \psi(0) =1$, we have $t=1$. Therefore, it follows that $\lambda^+_1(v_2) =
\lambda_1^+(v_1, v_2)$ and $\hat \psi_n \to \psi_1(v_2)$ in $W^{2, p}_{loc} (\mathbb{R}^d) \cap C(\mathbb{R}^d)$.

Thus we have $\hat \lambda_n \uparrow \lambda^+_1(v_2)$ and along a subsequence
$\hat \psi_n \to \psi_1(v_2)$ in $W^{2, p}_{loc}(\mathbb{R}^d \cap C(\mathbb{R}^d)$. Now combaining (\ref{eq2.1thm4.2}) and (\ref{eq2thm4.2}) we get the required representation. This completes the proof of the theorem.
\end{proof}
\begin{remark}\label{R1A}
Form the proof the Theorem~\ref{stochasticrepeigenfunction}, we conclude that
$\lambda_1^+(v_2) = \lambda^+(v_1, v_2)$ for any minimizing selector $v_1\in\mathcal{S}_1$ of the HJB equation ${\mathcal G}^{v_2}_1 \psi_1(v_2) = \lambda_1^+(v_2)\psi_1(v_2)$\,, where $\lambda_1^+(v_1, v_2)$ is the generalized principle eigenvalue of $\mathcal{L}^{v_1, v_2} + r_1(x, v_1(x), v_2(x))$\,. Similarly, $\lambda_2^+(v_1) = \lambda_2^+(v_1, v_2)$ for any minimizing selector $v_2\in\mathcal{S}_2$ of the HJB equation ${\mathcal G}^{v_1}_2 \psi_2(v_1) = \lambda_2^+(v_1)\psi_2(v_1)$\,, where $\lambda_2^+(v_1, v_2)$ is the generalized principle eigenvalue of $\mathcal{L}^{v_1, v_2} + r_2(x, v_1(x), v_2(x))$\,.
\end{remark}

Now we claim that $\lambda^+_1(v_2), \lambda^+_2(v_1) \geq 0$. If not, suppose that $\lambda^+_1(v_2) < 0$. Then from (\ref{ET4.2AA}), we deduce that $\psi_1(v_2)(x) \geq \min_{B_r}\psi_1(v_2)$ for all $x\in B_r^c$\,. Applying It\^{o}-Krylov formula and Fatou's lemma, from (\ref{ET4.1AA}) it is follows that
\begin{align*}
\psi_1(v_2)(x) &\geq
\mathbb{E}^{v_1, v_2} \Big[ e^{\int^T_0 (r_1(X(t), v_1(X(t)), v_2(X(t))) - \lambda^+_1(v_2)) dt }\psi_1(v_2)(X(T)) \Big]\\
&\geq \min_{B_r}\psi_1(v_2)\mathbb{E}^{v_1, v_2} \Big[ e^{\int^T_0 (r_1(X(t), v_1(X(t)), v_2(X(t))) - \lambda^+_1(v_2)) dt }\Big]\,.
\end{align*}
 Taking logarithm of  both sides, dividing by $T$ and letting $T\to\infty$, we get
\begin{equation}
\lambda^+_1(v_2) \geq \limsup_{T\to\infty}\frac{1}{T}\log \mathbb{E}^{v_1, v_2} \Big[ e^{\int^T_0 r_1(X(t), v_1(X(t)), v_2(X(t))) dt }\Big] \geq 0\,.
\end{equation} This is a contradiction. Thus,  $\lambda^+_1(v_2) \geq 0$. Similarly $\lambda^+_2(v_1) \geq 0$.

Now we show that the map $v_j \mapsto (\lambda^+_i(v_j) , \psi_i (v_j))$ is continuous in the topology of ${\mathcal S}_j$ for $i,j=1,2$\,. This result is useful in establishing the u.s.c. of a certain set-valued map ( to be be introduced soon), which in turn, will ensure the existence of a Nash equilibrium\,.
\begin{theorem}\label{ThmCont}
 Let Assumptions \ref{A1} and \ref{A2} hold. Then the  map  $v_j \mapsto (\lambda^+_i(v_j) , \psi_i (v_j)) $ from ${\mathcal S}_j$ to $\mathbb{R} \times  W^{2, p}_{loc}(\mathbb{R}^d) \cap C(\mathbb{R}^d)$ is continuous for $i,j = 1,2$\,.
\end{theorem}
\begin{proof} Take $i=1, j=2$. Let $v^n_2 \to v_2$ in the topology of stationary Markov strategies. From the above observation and (\ref{EEst1}), we get
\[
0 \leq \lambda^+_1 (v^n_2) \leq  \max\{\kappa_1 + \frac{\tilde{\alpha}}{\min_{\tilde{K}}\mathcal{V}}, \; \|r_1\|_{\infty}\}.
\] Now using Harnack inequality, see \cite[Corollary 8.21, p.199]{GilbargTrudinger}, and the interior estimates \cite[Theorem 9.11, p.235]{GilbargTrudinger}, we get for each bounded domain $D$,  there exists $n_0$ such that
\begin{equation}\label{eq2thm4.1}
\sup_{n \geq n_0} \|\psi_1 (v^n_2) \|_{2, p, D} < \infty.
\end{equation}
Hence, by a standard approximation procedure involving Sobolev imbedding (as in Theorem \ref{existenceeigenvalue}), we obtain the existence of $\psi \in W^{2, p}_{loc} (\mathbb{R}^d) \cap C(\mathbb{R}^d), p \geq 2, \psi > 0$
and a limit point $\lambda$ of $\lambda^+_1(v^n_2)$ satisfying
\begin{equation}\label{eq3thm4.1}
{\mathcal G}^{v_2}_1 \psi \ = \  \lambda \psi \  {\rm in} \  \mathbb{R}^d\,.
\end{equation}
Clearly
\[
\lambda \geq \lambda^+_1(v_2).
\]
Next we prove the reverse inequality. From Assumption \ref{A2}, we deduce that there exist a compact set $\mathcal{B}\, (\supset \tilde{K})$ and a constant $\theta\in (0, 1)$ such that for all large $n\in \NN$
\begin{itemize}
\item under Assumption \ref{A2}(i):\, $(\sup_{u_i\in U_i\,\, i = 1,2}r_1(x, u_1, u_2) - \lambda^+_1(v^n_2)) < \theta \gamma$ for all $x\in \mathcal{B}^c$
\item under Assumption \ref{A2}(ii): $(\sup_{u_i\in U_i\,\, i = 1,2}r_1(x, u_1, u_2) - \lambda^+_1(v^n_2)) < \theta \ell(x)$ for all $x\in \mathcal{B}^c$,.
\end{itemize} Let $r_0 > 0$ be such that $\mathcal{B}\subset B_{r_0}$. Applying It\^{o}-Krylov formula and Fatou's lemma, from (\ref{EA21}) and (\ref{EA22}), for any $(v_1, v_2)\in \mathcal{A}_1\times \mathcal{A}_2$ we deduce that
\begin{equation}\label{ET4.3AA}
\mathbb{E}^{v_1, v_2}_x \Big[ e^{\gamma\breve{\tau}_{r_0}}\sV(X(\breve{\tau}_{r_0})) \Big] \leq \sV(x)\,\,\,\text{and}\,\,\, \mathbb{E}^{v_1, v_2}_x \Big[ e^{\int_{0}^{\breve{\tau}_{r_0}}\ell(X(t)) dt} \sV(X(\breve{\tau}_{r_0})) \Big] \leq \sV(x)\,\,\,\forall\,\, x\in B_{r_0}^c\,.
\end{equation}
Thus, from Theorem \ref{stochasticrepeigenfunction}, for any minimizing selector $v_{1}^n$ of ${\mathcal G}^{v_2^n}_1 \psi_1(v_2^n) = \lambda^+_1(v_2^n) \psi_1(v_2^n)$, and $x\in B_{r_0}^c$,  it follows that
\begin{align}\label{ET4.3A}
\psi_1(v_2^n) (x) \ & =  \  \mathbb{E}^{v_1^n, v_2^n}_x \Big[ e^{\int^{\breve{\tau}_{r_0}}_0 (r_1(X(t), v_1^n(X(t)), v_2^n(X(t))) - \lambda^+_1(v_2^n) ) dt } \psi_1(v_2^n) (X(\breve{\tau}_{r_0})) \Big]\nonumber\\
&\leq\frac{\sup_{B_{r_0}}\psi_1(v_2^n)}{\inf_{B_{r_0}}\sV^{\theta}} \mathbb{E}^{v_1^n, v_2^n}_x \Big[ e^{\theta\breve{\tau}_{r_0}\gamma} \sV^{\theta}(X(\breve{\tau}_{r_0})) \Big]\nonumber\\
&\leq\frac{\sup_{B_{r_0}}\psi_1(v_2^n)}{\inf_{B_{r_0}}\sV^{\theta}} \left(\mathbb{E}^{v_1^n, v_2^n}_x \Big[ e^{\breve{\tau}_{r_0}\gamma} \sV(X(\breve{\tau}_{r_0})) \Big]\right)^{\theta}\quad \text{(by Jensen's inequality)}\nonumber\\
&\leq \hat{\kappa}_2 \sV^{\theta}(x)\quad \text{(by (\ref{ET4.3AA}))},
\end{align}
where one can choose the constant $\hat{\kappa}_2 >0$ independent of $n$ (by Harnack's inequality). This implies that $\psi \leq \hat{\kappa}_2  \sV^{\theta}$ (in the above calculations replacing $\gamma$ by $\ell$, it is easy to see that same estimate holds true under Assumption~\ref{A2}(ii)). Now for any minimizing selector $v_1$ of (\ref{ET4.1AA}), applying It\^{o}-Krylov formula from (\ref{eq3thm4.1}) for some $T>0$ we deduce that
\begin{equation*}
\psi(x) \leq \  \mathbb{E}^{v_1, v_2}_x \Big[ e^{\int^{\breve{\tau}_{r_0}\wedge T}_0 (r_1(X(t), v_1(X(t)), v_2(X(t))) - \lambda) dt } \psi(X(\breve{\tau}_{r_0}\wedge T))\Big]\,.
\end{equation*}
In view of (\ref{ET4.3AA}), since $\psi \leq \hat{\kappa}_2 \sV^{\theta}$, by the dominated convergence theorem letting $T\to \infty$, we get
\begin{align}\label{ET4.3B}
\psi(x) &\leq \  \mathbb{E}^{v_1, v_2}_x \Big[ e^{\int^{\breve{\tau}_{r_0}}_0 (r_1(X(t), v_1(X(t)), v_2(X(t))) - \lambda) dt } \psi(X(\breve{\tau}_{r_0}))\Big]
\nonumber\\
&\leq \  \mathbb{E}^{v_1, v_2}_x \Big[ e^{\int^{\breve{\tau}_{r_0}}_0 (r_1(X(t), v_1(X(t)), v_2(X(t))) - \lambda_{1}^{+}(v_2)) dt } \psi(X(\breve{\tau}_{r_0}))\Big]\,.
\end{align} Thus, from (\ref{ET4.2AA}) (for $i=1, j=2$) and (\ref{ET4.3B}), we have
\begin{equation}\label{ET4.3C}
(\psi_1(v_2) - \psi)(x) \geq \  \mathbb{E}^{v_1, v_2}_x \Big[ e^{\int^{\breve{\tau}_{r_0}}_0 (r_1(X(t), v_1(X(t)), v_2(X(t))) - \lambda_{1}^{+}(v_2)) dt } (\psi_1(v_2) - \psi)(X(\breve{\tau}_{r_0}))\Big]\,.
\end{equation}Let $\Tilde{\kappa}_{2} = \sup_{B_{r_0}}\frac{\psi_1(v_2)}{\psi}$. Hence (\ref{ET4.3C}) implies that $(\psi_1(v_2) - \Tilde{\kappa}_{2}\psi) \geq 0$ in $\mathbb{R}^d$, and for some $x_1 \in B_{r_0}$ we have $(\psi_1(v_2) - \Tilde{\kappa}_{2}\psi)(x_1) = 0$\,. Since $\lambda \geq \lambda^+_1(v_2)$,\, (\ref{eq1thm4.2}) and (\ref{eq3thm4.1}) give us
\[
{\mathcal L}_1^{v_1, v_2} (\psi_1(v_2)-  \Tilde{\kappa}_{2} \psi ) - (r_1(x, v_1(x), v_2(x)) - \lambda^+_1(v_2) )^{-}
(\psi_1(v_2)-  \Tilde{\kappa}_{2} \psi )  \leq 0\,.
\]
Thus, by the  strong maximum principle \cite[Theorem~9.6]{GilbargTrudinger}, we obtain  $\psi_1 (v_2) = \Tilde{\kappa}_{2} \psi$. But, we have $\psi_1 (v_2) (0) = \psi(0) = 1$, this gives $\Tilde{\kappa}_{2}=1$. Therefore, we deduce that $\psi_1 (v_2) = \psi$ and $\lambda^+_1(v_2) \geq
\lambda$. This, in particular, implies that $\lambda^+_1(v_2) = \lambda$. This proves the continuity of the map $v_2 \mapsto (\lambda^+_1(v_2) , \psi_1 (v_2))$ and the continuity of the other maps follows by analogous arguments.
\end{proof}
\begin{remark}\label{EStoRep1A}
For any $v\in \mathcal{S}_1$, by It\^{o}-Krylov formula, from (\ref{ET4.1AA}) we deduce that 
\begin{align}\label{EStoRep1AA}
\psi_1(v_2) (x) \  \leq &  \  \mathbb{E}^{v, v_2}_x \Big[ e^{\int^{\breve{\tau}_r\wedge \tau_n}_0 (r_1(X(t), v(X(t)), v_2(X(t))) - \lambda^+_1(v_2) ) dt } \psi_1(v_2) (X(\breve{\tau}_r\wedge \tau_n) \Big]\nonumber\\
 = & \mathbb{E}^{v, v_2}_x \Big[ e^{\int^{\breve{\tau}_r}_0 (r_1(X(t), v(X(t)), v_2(X(t))) - \lambda^+_1(v_2) ) dt } \psi_1(v_2) (X(\breve{\tau}_r) \Ind_{\{\breve{\tau}_r\leq \tau_n\}}\Big]\nonumber\\
 & +  \mathbb{E}^{v, v_2}_x \Big[ e^{\int^{\tau_n}_0 (r_1(X(t), v(X(t)), v_2(X(t))) - \lambda^+_1(v_2) ) dt } \psi_1(v_2) (X(\tau_n) \Ind_{\{\breve{\tau}_r\geq \tau_n\}}\Big]\,.
\end{align} Since $\psi_1(v_2) \leq \hat{\kappa}_2 \sV^{\theta}$ for some $\theta\in (0, 1)$ (see Theorem~\ref{ThmCont}, eq. (\ref{ET4.3A})), by mimicking the arguments as in the proof of \cite[Theorem~3.2]{SP21A}, it is easy to see that  
\begin{equation*}
\lim_{n\to \infty}\mathbb{E}^{v, v_2}_x \Big[ e^{\int^{\tau_n}_0 (r_1(X(t), v(X(t)), v_2(X(t))) - \lambda^+_1(v_2) ) dt } \psi_1(v_2) (X(\tau_n) \Ind_{\{\tau_n \leq \breve{\tau}_r\}}\Big] = 0\,.
\end{equation*} Thus, by monotone convergence theorem letting $n\to\infty$, from (\ref{EStoRep1AA}) we conclude that
\begin{align}\label{EStoRep1AB}
\psi_1(v_2) (x) \  \leq   \  \mathbb{E}^{v, v_2}_x \Big[ e^{\int^{\breve{\tau}_r}_0 (r_1(X(t), v(X(t)), v_2(X(t))) - \lambda^+_1(v_2) ) dt } \psi_1(v_2) (X(\breve{\tau}_r) \Big] \,.
\end{align}
\end{remark}

Next we show that for each $v_j \in \mathcal{S}_j$ the generalized principal eigenvalue $\lambda^+_i(v_j)$ is the optimal ergodic cost of Player $i$, i.e., $\lambda^+_i(v_j) = \lambda_i (v_j)$,\, $i,j=1,2$\,.
\begin{theorem}\label{RepPrinc1} Suppose that Assumptions \ref{A1} and \ref{A2} hold. Then for $i,j = 1,2$ we have
\[
\lambda^+_i(v_j) = \lambda_i (v_j)\,.
\]
\end{theorem}
\begin{proof}
From the Theorem \ref{existenceeigenvalue}, we have $\lambda^+_i(v_j) \leq \lambda_i (v_j)$\,. Now to prove the reverse inequality, we approximate the running costs in the following way:
\begin{itemize}
\item When the cost is bounded: let $\{\phi_{i,n}\}$ be a sequence of test functions such that $\phi_{i,n} = 1$ in $B_{n}$ and $\phi_{i,n} = 0$ in $B_{n+1}^{c}$. Since  $\|r_{i}\|_{\infty} < \delta$, it is possible to choose constants $\hat{\delta}_{i} > 0$ small enough such that $\|r_{i}\|_{\infty}+\hat{\delta}_{i} < \delta$. For $(x, u_{1}, u_{2})\in \RR^d\times U_{1}\times U_{2}$, set $$r_{i,n}(x, u_{1}, u_{2}) = \phi_{n}(x)r_i(x, u_{1}, u_{2}) + (1-\phi_{n}(x))(\|r_{i}\|_{\infty} + \hat{\delta}_{i}),\,\,\,\forall\,\,\, n\in\mathbb{N}.$$
\item When the cost is unbounded: For $(x, u_{1}, u_{2})\in \RR^d\times U_{1}\times U_{2}$ we define
\begin{equation*}
r_{i,n}(x, u_{1}, u_{2}) = r_{i}(x, u_{1}, u_{2}) + \frac{1}{2}\left(\ell(x) - r_{i}(x, u_{1}, u_{2})\right)^{+}\Ind_{\{B_n^c\}}\,.
\end{equation*}
\end{itemize}
It is easy to see that for $r_{i,n}$ satisfies (\ref{EA23}) for $i =1,2$\,.

Now from Theorem \ref{existenceeigenvalue}, for each $n\in\mathbb{N}$,  there exists $(\lambda_{1,n}^+(v_2), \psi_{1,n}(v_2))\in \mathbb{R} \times W_{loc}^{2, p}(\mathbb{R}^{d})\cap C(\mathbb{R}^d),$ $ 2\leq p < \infty,\, \psi_{1,n}(v_2) > 0,$ satisfying
\interdisplaylinepenalty=0
\begin{align}\label{ET4.4A}
\lambda_{1,n}^+(v_2) \psi_{1,n}(v_2)(x) =& \inf_{v_1 \in V_1} [ {\mathcal L}_1^{v_1, v_2} \psi_{1,n}(v_2) + r_{1,n} (x, v_1, v_2(x)) \psi_{1,n}(v_2) ]\,, \quad\text{with}\,\,\,\psi_{1,n}(0) = \, 1\,,
\end{align}
and
\begin{eqnarray}\label{ET4.4B}
\lambda_{1,n}^+(v_2) \leq \inf_{x\in\mathbb{R}^d}\inf_{v_1\in\mathcal{A}_1}\limsup_{T \to \infty} \frac{1}{ T} \log \mathbb{E}^{v_{1}, v_{2}}_x \Big[ e^{\int^{T}_0 r_{1,n}(X(t), v_{1}(t), v_{2}(X(t)))dt }\Big].
\end{eqnarray} It is clear from our construction that there exists a compact set $\mathcal{K}$ containing $\tilde{K}$ such that $\displaystyle{\inf_{(u_{1}, u_{2})\in U_{1}\times U_{2}}r_{1,n}(x,u_{1},u_{2}) - \lambda_{1,n}^+(v_2)\geq 0}$ for all $x\in\mathcal{K}^{c}$. Under Assumption \ref{A2}(i) one can take $\mathcal{K} = \overline{B}_{n+1}$ and under Assumption \ref{A2}(ii) since $r_{1,n}$ is unbounded and it satisfies (\ref{EA23}) one can suitably choose $\mathcal{K}$ which satisfies the above inequality. Let $$\breve{\tau}(\mathcal{K}) = \inf\{t\geq 0: X(t)\in \mathcal{K}\}.$$
%Without loss of generality we assume that $\hat{K}_{1}\supset \mathcal{B}_{1}$.
Applying It\^{o}-Krylov fromula and Fatous lemma, for any minimizing selector $\hat{v}_1$ of (\ref{ET4.4A}), it follows that
\begin{eqnarray*}
\psi_{1,n}(v_2)(x) & \geq & E^{\hat{v}_{1}, v_{2}}_x \Big[ e^{\int^{\breve{\tau}(\mathcal{K})}_0 (r_{1,n}(X(t), \hat{v}_{1}(X(t)), v_{2}(X(t))) - \lambda_{1,n}^+(v_2)) dt } \psi_{1,n}(v_2)(X(\breve{\tau}(\mathcal{K}))) \Big],\nonumber\\
&\geq &\inf_{\mathcal{K}}\psi_{1,n}(v_2),\,\,\, \forall\,\,\, x\in \mathcal{K}^{c}.
\end{eqnarray*}
Thus, by another application of It$\hat{\rm o}$-Krylov's formula and Fatou's lemma, we deduce that
\begin{eqnarray*}
\psi_{1,n}(v_2)(x) & \geq & E^{\hat{v}_{1}, v_{2}}_x \Big[ e^{\int^{T}_0 (r_{1,n}(X(t), \hat{v}_{1}(X(t)), v_{2}(X(t))) - \lambda_{1,n}^+(v_2)) dt } \psi_{1,n}(v_2)(X(T)) \Big],\nonumber\\
&\geq &\inf_{\mathcal{K}}\psi_{1,n}(v_2)E^{\hat{v}_{1}, v_{2}}_x \Big[ e^{\int^{T}_0 (r_{1,n}(X(t), \hat{v}_{1}(X(t)), v_{2}(X(t))) - \lambda_{1,n}^+(v_2)) dt }\Big].
\end{eqnarray*}
Taking logarithm on both sides, dividing by $T$ and then letting $T\to\infty$, we get
\begin{eqnarray}\label{ET4.4C}
\lambda_{1,n}^+(v_2) &\geq & \limsup_{T \to \infty} \frac{1}{ T} \log E^{\hat{v}_{1}, v_{2}}_x \Big[ e^{\int^{T}_0 r_{1,n}(X(t), \hat{v}_{1}(X(t)), v_{2}(X(t)))dt} \Big],\nonumber\\
&\geq & \limsup_{T \to \infty} \frac{1}{T} \log E^{\hat{v}_{1}, v_{2}}_x \Big[ e^{\int^{T}_0 r_{1}(X(t), \hat{v}_{1}(X(t)), v_{2}(X(t))) dt } \Big].
\end{eqnarray}
As in Theorem \ref{existenceeigenvalue}, using Harnack's inequality and Sobolev estimate from (\ref{ET4.4A}), one can clearly see that $\psi_{1,n}(v_2)$ is uniformly bounded in $W_{loc}^{2,p}(\mathbb{R}^{d}),$ $2\leq p < \infty.$ Thus, along a suitable subsequence $\{\psi_{1,n}(v_2)\}$ converges weakly in $W_{loc}^{2,p}(\mathbb{R}^{d}),$ $2\leq p < \infty,$ to some $\psi_{1,*}(v_2)\in W_{loc}^{2,p}(\mathbb{R}^{d}),$ $2\leq p < \infty,$ and strongly in $C_{loc}^{1,\hat{\alpha}}(\mathbb{R}^{d}),$ $\hat{\alpha}\in (0, 1).$ It is clear from (\ref{ET4.4B}) and (\ref{ET4.4C}), that $\{\lambda_{1,n}^+(v_2)\}$ is a bounded sequence. Thus, along a further subsequence it converges to a constant $\lambda_{1,*}(v_2)$. Now as in Theorem \ref{existenceeigenvalue}, letting $n\to \infty$ in (\ref{ET4.4A}), we get $(\lambda_{1,*}(v_2), \psi_{1,*}(v_2))\in \mathbb{R}\times W_{loc}^{2,p}(\mathbb{R}^{d}),$ $2\leq p < \infty$, satisfies
\begin{align}\label{ET4.4D}
\lambda_{1,*}(v_2)\psi_{1,*}(v_2) =& \inf_{v_1 \in V_1} \left[ {\mathcal L}_1^{v_1, v_2} \psi_{1,*}(v_2) + r_{1} (x, v_1, v_2(x)) \psi_{1,*}(v_2) \right]\nonumber \\
\psi_{1,*}(v_2)(0) =& 1.
\end{align} Following the argument as in Theorem \ref{ThmCont} (see ( \ref{ET4.3A})), one can show that $\psi_{1,n}(v_2)\leq\hat{\kappa}_{2}\sV^{\theta}$, uniformly in $n$ for some constant $\hat{\kappa}_{2} > 0$ and $\theta\in (0, 1)$. This implies that, the limit $\psi_{1,*}(v_2)\leq\hat{\kappa}_{2}\sV^{\theta}$. Let $v_1\in  \mathcal{S}_1$ be a minimizing selector of (\ref{ET4.1AA})\,. Now, by the arguments as in Remark~\ref{EStoRep1A}, for each large $n\in \NN$, we have
\begin{equation}\label{ELimRep1}
\psi_{1,n}(v_2)(x) \leq E^{v_1, v_{2}}_x \Big[ e^{\int^{\breve{\tau}_{r}}_0 (r_{1,n}(X(t), v_1(X(t)), v_{2}(X(t))) - \lambda_{1,n}^+(v_2)) dt } \psi_{1,n}(X(\breve{\tau}_{r})) \Big],\,\,\, \forall\,\,\, x\in B_{r}^{c},
\end{equation}
 for some $r>0$\,. Since $\psi_{1,n}(v_2)\leq\hat{\kappa}_{2}\sV^{\theta}$ ( uniformly in $n$ ), in view of estimates as in (\ref{ET4.3AA}), by the  dominated convergence theorem letting $n\to \infty$ from (\ref{ELimRep1}) we deduce that
\begin{equation}\label{ELimRep2}
\psi_{1,*}(v_2)(x) \leq E^{v_1, v_{2}}_x \Big[ e^{\int^{\breve{\tau}_{r}}_0 (r_{1}(X(t), v_1(X(t)), v_{2}(X(t))) - \lambda_{1,*}(v_2)) dt } \psi_{1,*}(X(\tau_{1}^{c})) \Big],
\end{equation} for all $x\in B_r^c$\,.

From (\ref{ET4.4C}), it is easy to see that $\lambda_{1,*}(v_2) \geq \lambda_1(v_2)$\,. To complete the proof, we have to show that $\lambda_{1}^{+}(v_2)\geq \lambda_{1,*}(v_2).$ If not, let $\lambda_{1}^{+}(v_2) < \lambda_{1,*}(v_2).$ From Theorem \ref{stochasticrepeigenfunction}, we have for $x\in B_{r}^{c}$
\begin{eqnarray}\label{EContRep1A}
\psi_{1}(v_2)(x) & = & E^{v_{1}, v_{2}}_x \Big[e^{ \int^{\breve{\tau}_{r}}_0 (r_{1}(X(t), v_{1}(X(t)), v_{2}(X(t)))- \lambda_{1}^{+}(v_2))dt}\psi_{1}(v_2)(X(\breve{\tau}_{r}))\Big]\nonumber\\
&\geq & E^{v_{1}, v_{2}}_x \Big[e^{ \int^{\breve{\tau}_{r}}_0 (r_{1}(X(t), v_{1}(X(t)), v_{2}(X(t)))- \lambda_{1,*}(v_2))dt}\psi_{1}(v_2)(X(\breve{\tau}_{r}))\Big].
\end{eqnarray}
From (\ref{ELimRep2}) and (\ref{EContRep1A}), it follows that
\begin{equation*}
(\psi_{1}(v_2) - \psi_{1,*}(v_2))(x) \geq E^{v_{1}, v_{2}}_x \Big[e^{ \int^{\breve{\tau}_{r}}_0 (r_{1}(X(t), \hat{v}_{1}(X(t)), \hat{v}_{2}(X(t)))- \lambda_{1,*}(v_2))dt}(\psi_{1} - \psi_{1,*})(X(\breve{\tau}_{r}))\Big].
\end{equation*}This implies that $(\psi_{1}(v_2) - \psi_{1,*}(v_2))(x) \geq 0$ for all $x\in \mathbb{R}^{d}$, if it holds in $B_{r}.$ Now multiplying $\psi_{1,*}(v_2)$ by a suitable positive constant (say, $\hat{k}_{1}=\displaystyle{\inf_{B_{r}}\frac{\psi_{1}(v_2)}{\psi_{1,*}(v_2)}}$), we obtain that $(\psi_{1}(v_2) - \tilde{\psi}_{1,*}(v_2))(x) \geq 0$ in $B_{r}$ and it attains its minimum value $0$ in $B_{r}$, where $\tilde{\psi}_{1,*}(v_2) = \hat{k}_{1}\psi_{1,*}(v_2).$ It is clear that $\tilde{\psi}_{1,*}(v_2)$ also satisfies (\ref{ET4.4C}). Thus, from (\ref{ET4.1AA}) and (\ref{ET4.4C}) (for $\tilde{\psi}_{1,*}$), we obtain
\begin{align*}
& {\mathcal L}_1^{v_1, v_2}(\psi_{1}(v_2) - \tilde{\psi}_{1,*}(v_2)) - (r_{1}(x, v_{1}(x), v_{2}(x)) - \lambda_{1,*}(v_2))^{-}(\psi_{1} - \tilde{\psi}_{1,*})\nonumber\\
&\leq - (r_{1}(x, \hat{v}_{1}(x), \hat{v}_{2}(x)) - \lambda_{1,*}(v_2))^{+}(\psi_{1}(v_2) - \tilde{\psi}_{1,*}(v_2)) \leq 0 \,.
\end{align*} Thus, by an application of the  strong maximum principle as in \cite[Theorem 9.6]{GilbargTrudinger}, we have $\psi_{1}(v_2) = \tilde{\psi}_{1,*}(v_2).$ Since $\psi_{1}(v_2)(0) = \psi_{1,*}(v_2)(0) = 1$, we obtain $\psi_{1}(v_2) = \psi_{1,*}(v_2).$ Hence, from (\ref{ET4.1AA}) and (\ref{ET4.4C}), we deduce that
\[
\lambda_{1,*}(v_2)\psi_{1,*}(v_2) \leq \lambda_{1}^{+}(v_2)\psi_{1,*}(v_2).
\] Since $\psi_{1,*}(v_2)>0$, we conclude that $\lambda_{1}^{+}(v_2) \geq \lambda_{1,*}(v_2).$ This contradicts the fact that $\lambda_{1}^{+}(v_2) < \lambda_{1,*}(v_2).$ Therefore we obtain $\lambda_{1}^{+}(v_2) \geq \lambda_{1,*}(v_2)$. This completes the proof of the theorem.
\end{proof}
\begin{remark}\label{ER1B}
By closely following the arguments as in the proof of the Theorem~\ref{RepPrinc1}, one can conclude that for any $(v_1, v_2)\in \mathcal{S}_1\times\mathcal{S}_2$ the generalized principle eigenvalue $\lambda_{i}^+(v_1, v_2)$ of $\mathcal{L}^{v_1, v_2} + r_i(x, v_1(x), v_2(x))$, satisfies $\lambda_{i}^+(v_1, v_2) = \rho_i(x, v_1, v_2)$ for $i = 1,2$ and $x\in\RR^d$\,.
\end{remark}
\section{Existence of  Nash equilibrium} In this section using Fan's fixed point theorem, we establish the existence of Nash equilibria in the space of stationary Markov strategies. Also, exploiting the stochastic representation of the principal eigenfunctions of the associated coupled HJB equation we completely characterize  all possible Nash equilibria in the space of stationary Markov strategies.

Let $(v_1, v_2) \in \mathcal{S}_1 \times \mathcal{S}_2.$ Define
\begin{equation}\label{optimalresponseset}
N(v_1, v_2) \ = \ N_1(v_2) \times N_2(v_1),
\end{equation}
where
\begin{equation*}\label{optimalrespeonse1}
N_1(v_2) \ = \ \Big\{v^*_1 \in \mathcal S_1 \mid F_1 (x, v^*_1(x), v_2(x)) =
\inf_{v_1 \in V_1} F_1(x,  v_1, v_2(x)) \ {\rm a.e.} \  x \Big\},
\end{equation*}
\begin{equation*}
F_1(x, v_1, v_2(x) ) \ = \ \langle b(x, v_1, \hat v_2(x)) , \nabla \psi_1(v_2) \rangle + r_1(x, v_1, v_2(x))\psi_1(v_2), \,\,\, x \in \mathbb{R}^{d}, \,\, v_1 \in V_1,\,\, v_2 \in \mathcal S_2
\end{equation*} and
\begin{equation*}\label{optimalrespeonse2}
N_2(v_1) \ = \ \Big\{v^*_2 \in \mathcal S_2 \mid F_2 (x, v_1(x), v^*_2(x)) =
\inf_{v_2 \in V_2} F_2(x, v_1(x), v_2) \ {\rm a.e.} \  x \Big\},
\end{equation*} where
\begin{align*}
F_2(x, v_1(x), v_2 ) = \langle b(x, v_1(x) , v_2) , \nabla \psi_2(v_1) \rangle + r_2(x, v_1(x), v_2)\psi_1(v_2),\,\,\, x \in \mathbb{R}^{d},\,\, v_2 \in V_2,\,\, \hat v_1 \in \mathcal S_1\,.
\end{align*}
\vspace{.1in}
By a standard measurable selection theorem (see, \cite{Benes}), it is clear that $N_1(v_2)$ is nonempty. Also, it is easy to see that $N_1(v_2)$ is convex. Under the topology of $\mathcal S_1,$ one can show that $N_1(v_2)$ is closed in $\mathcal S_1,$ hence compact. Similarly, one can show that $N_2(v_1)$ is nonempty, compact, convex subset of $\mathcal S_2.$ Therefore  $N(v_1, v_2)$ is nonempty, convex and compact subset of $\mathcal S_1 \times \mathcal S_2$. To establish the existence of a Nash equilibrium, we next prove the upper semi-continuity (u.s.c.) of the map
$(v_1, v_2) \mapsto N(v_1, v_2) $ from $\mathcal S_1 \times \mathcal S_2 \to
2^{\mathcal S_1 \times \mathcal S_{2}}$. In order to do so we we impose some additive structure on the drift of the state dynamics and the running cost function, which is known as (ADAC) condition, given as follows.
\begin{assumption}\label{A3}
 We assume that $\bar{b}: \mathbb{R}^d \times U_1\times U_2 \to \RR^d$ and $\bar{r}_i : \RR^d \times U_1 \times U_2 \to \RR_{+},\, i = 1, 2$\,, admit the following additive structures given by
\begin{align*}
\bar{b}(x, u_1, u_2) &= \bar{b}_1(x, u_1) + \bar{b}_2(x, u_2)\\
\bar{r}_i(x, u_1, u_2) &= \bar{r}_{i,1}(x, u_1) + \bar{r}_{1,2}(x, u_2)
\end{align*}
where $\bar{b}_1, \bar{b}_2, \bar{r}_{i,1}, \bar{r}_{i,2}$ satisfy the conditions in Assumption \ref{A1}(i)-(ii)\,.
\end{assumption}
\vspace{.1in}
Next lemma shows that our set valued map $(v_1, v_2) \mapsto N(v_1, v_2)$ is upper semi-continuous\,.
\begin{lemma}\label{usc1} Let Assumptions \ref{A1} - \ref{A3} hold. Then the map
$(v_1, v_2) \mapsto N(v_1, v_2)$ from $\mathcal S_1 \times \mathcal S_2 \to
2^{\mathcal S_1 \times \mathcal S_2}$ is u.s.c.
\end{lemma}
\begin{proof} Consider a sequence $\{(v^n_1, v^n_2)\}_{n}$ in $\mathcal S_1 \times \mathcal S_2$ such that $(v^n_1, v^n_2) \to (v_1, v_2) \in\mathcal S_1 \times \mathcal S_2.$ Choose $\hat{v}^n_1 \in N_1(v^n_2), n \geq 1$. Since $\mathcal S_1$ is compact,
there exists a subsequence (denoting by the same notation without any loss of generality) $\{\hat{v}^n_1\}$ such that $\hat{v}^n_1 \to \hat{v}_1$ for some $\hat{v}_1 \in \mathcal S_1$. Then $(\hat{v}^n_1, v^n_2) \to (\hat{v}_1, v_2)$ in $\mathcal S_1\times\mathcal S_2$. In view of of Assumption \ref{A3}, the continuity results as in Theorem~\ref{ThmCont} and the topology of $\mathcal S_i, i =1,2$, we deduce that
$$\langle b(x, \hat{v}^n_1(x), v^n_2(x)),\nabla \psi_1(v^n_2) \rangle
+ r_1(x, \hat{v}^n_1(x), v^n_2(x))\psi_1(v^n_2)$$
converges weakly in $ L_{loc}^{2} (\mathbb{R}^{d})$ to
$$\langle b(x, \hat{v}_1(x), v_2(x)),\nabla \psi_1(v_2) \rangle
+  r_1(x, \hat{v}_1(x), v_2(x))\psi_1(v_2).$$
Thus, by Banach-Saks theorem \cite{ON84}, there exists a subsequence of the former whose convex combinations converges strongly in $L_{loc}^{2} (\mathbb{R}^{d})$ to the latter. Therefore, along a suitable subsequence of the convergent sequence of convex combinations (without any loss of generality denoting by the same notation), it follows that
\begin{equation}\label{eq1usc}
\lim_{n \to \infty} F_{1}(x, \hat{v}^n_1(x), v^n_2(x))
= F_{1}(x, \hat{v}_1(x), v_2(x)), \ {\rm a.e.\ in}\  x.
\end{equation}
By analogous arguments, for any fixed $\hat{\bar{v}}_1 \in  \mathcal{S}_1,$, we have
\begin{equation}\label{eq2usc}
\lim_{n \to \infty} F_{1}(x,  \hat{\bar{v}}_1(x), v^n_2(x))
= F_{1}(x, \hat{\bar{v}}_1(x), v_2(x)), \ {\rm a.e.\ in}\ x.
\end{equation}
Since $\hat{v}^n_1 \in N_1(v^n_2),$ from the definition of the set $N_1(v^n_2)$ it is easy to see that
\[
F_{1}(x,  \hat{\bar{v}}_1(x), v^n_2(x))  \geq F_{1}(x,  \hat{v}^n_1(x), v^n_2(x)), \quad \text{for all}\,\,\, n \geq 1\,.
\]
Thus, from (\ref{eq1usc}) and (\ref{eq2usc}), we obtain
\[
F_{1}(x,  \hat{\bar{v}}_1(x), v_2(x))  \geq F_{1}(x,  \hat{v}_1(x), v_2(x)) , \quad\text{for any}\,\,\, \hat{\bar{v}}_1 \in \mathcal{S}_1.
\]
This implies that  $\hat{v}_1 \in N_1(v_2)$. By similar argument, one can show that if $\hat{v}^n_2 \in N_2(v^n_1)$ and $\hat{v}^n_2 \to \hat{v}_2$ in $\mathcal{S}_2$ then $\hat{v}_2 \in N_2(v_1)$. This proves that the set valued map is u.s.c. 
\end{proof}
In view of the u.s.c. of the above set valued map, using Fan's fixed point theorem, we now establish the existence of Nash equilibrium in the space of stationary Markov strategies.
 \begin{theorem}\label{maintheorem1} Let Assumptions \ref{A1} - \ref{A3} hold. Then there exists $(v^*_1, v^*_2) \in {\mathcal S}_1 \times {\mathcal S}_2$ such that
\begin{equation*}
\lambda^+_1 (v^*_2) =  \lambda^+_1(v^*_1, v^*_2)\quad\text{and}\quad   \lambda^+_2 (v^*_1) =  \lambda^+_2(v^*_1, v^*_2)\,.
\end{equation*}
In particular, we have $(v^*_1, v^*_2) \in {\mathcal S}_1 \times {\mathcal S}_2$ is a Nash equilibrium.
 \end{theorem}
 \begin{proof} From Lemma~\ref{usc1}, we know that the set valued map $(v_1, v_2) \mapsto N(v_1, v_2)$ from $\mathcal S_1 \times \mathcal S_2 \to
2^{\mathcal S_1 \times \mathcal S_2}$ is u.s.c. Thus, by Fan's fixed point theorem \cite{Fan}, there exists a fixed point $(v^*_1, v^*_2) \in \mathcal S_1 \times \mathcal S_2$, of the map $(v_1, v_2) \mapsto N(v_1, v_2)$, i.e., $(v^*_1, v^*_2) \in N (v^*_1, v^*_2)\,.$ Therefore, it follows that $(\lambda_{1}^+(v^*_2),\psi_1(v^*_2)), (\lambda_{2}^+(v^*_1),\psi_2(v^*_1)) \in \RR_+ \times W^{2, p}_{loc} (\mathbb{R}^d) \cap C(\mathbb{R}^d), p \geq 2,$ satisfy the following coupled HJB equations
\begin{align}\label{CoupledHJB1}
\lambda^+_1(v_2^*) \psi_1(v_2^*)(x) \ = \ {\mathcal G}^{v_2^*}_1 \psi_1(v_2^*)(x) \ = \  {\mathcal L}^{v_1^*, v_2^*} \psi_1(v_2^*)(x) + r_{1}(x, v_1^*(x), v_2^*(x))\psi_1(v_2^*)(x)\,,
\end{align} and
\begin{align}\label{CoupledHJB2}
 \lambda^+_2(v_1^*) \psi_2(v_1^*)(x) \ = \ {\mathcal G}^{v_1^*}_2 \psi_2(v_1^*)(x) \ = \  {\mathcal L}^{v_1^*, v_2^*} \psi_2(v_1^*)(x) + r_{2}(x, v_1^*(x), v_2^*(x))\psi_2(v_1^*)(x)\,.
\end{align}
From Remark~\ref{R1A} (also see Theorem~\ref{stochasticrepeigenfunction}), it is easy to see that
\begin{equation*}
\lambda^+_1 (v^*_2) =  \lambda^+_1(v^*_1, v^*_2)\quad\text{and}\quad   \lambda^+_2 (v^*_1) =  \lambda^+_2(v^*_1, v^*_2)\,.
\end{equation*}
Therefore, in view of Theorem~\ref{RepPrinc1} and Remark~\ref{ER1B}, we conclude that
\begin{eqnarray*}
\rho_1 (x, v^*_1, v^*_2) \leq \rho_1 (x, v_1, v^*_2) \quad \text{and}\quad \rho_2 (x, v^*_1, v^*_2) \leq \rho_2 (x, v^*_1, v_2)\,,
\end{eqnarray*} for all $v_1 \in {\mathcal A}_1, v_2 \in {\mathcal A}_2$ and $x \in \mathbb{R}^d.$ This completes the proof of the theorem.
 \end{proof}
%%%%%%%%%%%%%%%%%%%%%%%%%%%%%%%%%%%%%%%%%%%%%%%%%%%%%%%%%%%%%%%%%%%%%%%%%%%%%
In the above theorem we have shown the existence of a  Nash equilibrium in the space of stationary Markov strategies. Conversely,  we now prove that if there exists a Nash equilibrium $(\bar{v}_{1}^{*}, \bar{v}_{2}^{*})\in \mathcal{S}_1\times \mathcal{S}_2$, then $(\bar{v}_{1}^{*}, \bar{v}_{1}^{*})$ is a pair of minimizing selectors of the associated coupled HJB equation\,.
\begin{theorem}
Suppose that Assumptions \ref{A1}- \ref{A3} hold. Then, if $(\bar{v}_{1}^{*}, \bar{v}_{2}^{*})\in \mathcal{S}_{1}\times\mathcal{S}_{2}$ is a Nash equilibrium, i.e.,
\begin{eqnarray*}
\rho_{1}(x,\bar{v}^*_1, \bar{v}^*_2) & \leq & \rho_{1}(x,\bar{v}_1, \bar{v}^*_2) , \ \forall \ \bar{v}_1 \in {\mathcal A}_1, \, x \in \mathbb{R}^d, \\
\rho_{2}(x,\bar{v}^*_1, \bar{v}^*_2) & \leq & \rho_{2}(x,\bar{v}^*_1, \bar{v}_2) , \ \forall \ \bar{v}_2 \in {\mathcal A}_2, \, x \in \mathbb{R}^d,
\end{eqnarray*} then $(\bar{v}_{1}^{*}, \bar{v}_{2}^{*})$ is a pair of  minimizing selector of the corresponding coupled HJB equation
\begin{equation}\label{EEHJB1}
\lambda^+_1(\bar{v}_2^*) \psi_1(\bar{v}_2^*)(x) \ = \ {\mathcal G}^{\bar{v}_2^*}_1 \psi_1(\bar{v}_2^*)(x)\,.
\end{equation}
\begin{equation}\label{EEHJB2}
\lambda^+_2(\bar{v}_1^*) \psi_2(\bar{v}_1^*)(x) \ = \ {\mathcal G}^{\bar{v}_1^*}_2 \psi_2(\bar{v}_1^*)(x)\,.
\end{equation}
\end{theorem}
\begin{proof} By limiting arguments as in Theorem \ref{existenceeigenvalue}, for the given pair $(\bar{v}_{1}^{*}, \bar{v}_{2}^{*})\in \mathcal{S}_{1}\times\mathcal{S}_{2},$ one can prove that there exists a principal eigenpair $(\lambda_{1}^+(\bar{v}_{1}^{*},\bar{v}_{2}^{*}), \psi_{1}(\bar{v}_{1}^{*},\bar{v}_{2}^{*}))\in \mathbb{R}_+\times W_{loc}^{2,p}(\mathbb{R}^{d}),$ $\infty > p \geq 2,$ with $\psi_{1}(\bar{v}_{1}^{*},\bar{v}_{2}^{*}) > 0$ satisfying the following
\begin{eqnarray}\label{eigenvaluehjb11characterization}
\lambda_{1}^+(\bar{v}_{1}^{*},\bar{v}_{2}^{*})\psi_{1}(\bar{v}_{1}^{*},\bar{v}_{2}^{*}) & = &  \mathcal{L}^{\bar{v}_{1}^{*},\bar{v}_{2}^{*}}\psi_{1}(\bar{v}_{1}^{*},\bar{v}_{2}^{*}) + r_{1}(x, \bar{v}_{1}^{*}(x),\bar{v}_{2}^{*}(x)) \psi_{1}(\bar{v}_{1}^{*},\bar{v}_{2}^{*})\nonumber \\
\psi_{1}^{\bar{v}_{1}^{*},\bar{v}_{2}^{*}}(0) &=& 1.
\end{eqnarray} From Remark \ref{ER1B}, we deduce that $\lambda_{1}^+(\bar{v}_{1}^{*},\bar{v}_{2}^{*}) = \rho_{1}(x,\bar{v}_{1}^{*},\bar{v}_{2}^{*})$\,. By similar argument as in Theorem~\ref{stochasticrepeigenfunction},we have
\begin{equation}\label{EEigenChar1A}
\psi_1(\bar{v}_{1}^{*},\bar{v}_{2}^{*})(x) \  =  \  \mathbb{E}^{\bar{v}_{1}^{*},\bar{v}_{2}^{*}}_x \Big[ e^{\int^{\breve{\tau}_r}_0 (r_1(X(t), \bar{v}_{1}^{*}(X(t),\bar{v}_{2}^{*}(X(t))) - \lambda^+_1(\bar{v}_{1}^{*},\bar{v}_{2}^{*})) dt } \psi_1(\bar{v}_{1}^{*},\bar{v}_{2}^{*}) (X(\breve{\tau}_r) \Big]\,,
\end{equation}
for some $r > 0$\,. In view of Theorem~\ref{existenceeigenvalue}, for given $\bar{v}_{2}^{*} \in {\mathcal S}_2$,  there exists a principal eigenpair
$(\lambda^+_1(\bar{v}_{2}^{*}), \psi_{1}(\bar{v}_{2}^{*}))\in \mathbb{R}_+ \times W_{loc}^{2, p}(\mathbb{R}^{d}),$\, $\psi_{1}(\bar{v}_{2}^{*}) > 0,$\, $\infty > p \geq 2,$ satisfying
\interdisplaylinepenalty=0
\begin{eqnarray}\label{Eeigencoupchar1}
\lambda^+_1(\bar{v}_{2}^{*})  \psi_{1}(\bar{v}_{2}^{*}) =
\mathcal{G}_{1}^{\bar{v}_{2}^{*}}\psi_{1}(\bar{v}_{2}^{*}) \quad \text{with}\quad \psi_{1}(\bar{v}_{2}^{*})(0) = 1\,.
\end{eqnarray} Remark \ref{R1A} implies that for any minimizing selector $\tilde{v}_{1}^{*}\in\mathcal{S}_{1}$ of (\ref{Eeigencoupchar1}), $\lambda_1^+(\bar{v}_{2}^{*}) = \rho_1(x,\tilde{v}_{1}^{*}, \bar{v}_{2}^{*}).$ From (\ref{Eeigencoupchar1}), it is easy to see that
\begin{align}\label{Eeigencoupchar2}
\lambda^+_1(\bar{v}_{2}^{*})  \psi_{1}(\bar{v}_{2}^{*}) &\leq \mathcal{L}^{\bar{v}_{1}^{*}, \bar{v}_{2}^{*}}\psi_{1}(\bar{v}_{2}^{*}) + r_{1}(x, \bar{v}_{1}^{*}(x), \bar{v}_{2}^{*}(x)) \psi_{1}(\bar{v}_{2}^{*})\,.
\end{align} By It$\hat{\rm o}$-Krylov formula, as in Theorem~\ref{existenceeigenvalue}, we obtain $\lambda^+_1(\bar{v}_{2}^{*}) \leq \rho_1(x, \bar{v}_{1}^{*}, \bar{v}_{2}^{*}).$ 
But we already have $\rho_{1}(x, \bar{v}^*_1, \bar{v}^*_2)\leq \rho_{1}(x, \bar{v}_1, \bar{v}^*_2) , \ \forall \ \bar{v}_1 \in {\mathcal A}_1, \, x \in \mathbb{R}^d.$ Therefore we get $\lambda_1^+(\bar{v}_{2}^{*}) = \rho_1(x,\tilde{v}_{1}^{*}, \bar{v}_{2}^{*}) = \rho_1(x,\bar{v}_{1}^{*}, \bar{v}_{2}^{*}).$ Following the proof of the Theorem~\ref{stochasticrepeigenfunction}, we get
\begin{equation*}
\psi_{1}(\bar{v}_{1}^{*})(x) \leq E^{\bar{v}_{1}^{*}, \bar{v}_{2}^{*}}_x \Big[ e^{\int^{\breve{\tau}_{r}}_0 (r_{1}(X(t), \bar{v}_{1}^{*}(X(t)), \bar{v}_{1}^{*}(X(t))) - \lambda_1^+(\bar{v}_{2}^{*})) dt } \psi_{1}(\bar{v}_{1}^{*})(X(\breve{\tau}_r)) \Big]\,.
\end{equation*}
Now applying the  maximum principle as in Theorem \ref{ThmCont},  one can deduce that $\psi_{1}(\bar{v}_{2}^{*}) = \psi_{1}(\bar{v}_{1}^{*},\bar{v}_{2}^{*}).$ Thus, from (\ref{EEHJB1}) and (\ref{eigenvaluehjb11characterization}), it follows that $\bar{v}_{1}^{*}$ is a minimizing selector of (\ref{EEHJB1}). By similar arguments one can show that $\bar{v}_{2}^{*}$ is a minimizing selector of (\ref{EEHJB2})\,. This completes the proof of the theorem\,.
\end{proof}
%%%%%%%%%%%%%%%%%%%%%%%%%%%%%%%%%%%%%%%%%%%%%%%%%%%%%%%%%%%%%%%%%%%%%%%%%%%%%%
\begin{appendix}
\section{}
In this section we state some important results which we have used in our proofs. First we recall a version of of the nonlinear Krein--Rutman theorem from \cite{A18}.
%%%%%%%%%%%%%%%%%%%%%%%%%%%%%%%%%%%%%%%%%%%%%%%%%%%%%%%%%%%%%%%%%%%%%%%%%%%%%%%%
\begin{theorem}\label{TAA.1}
Let $\mathfrak{C}$ be a nonempty closed cone in an ordered Banach space $\mathfrak{X}$ satisfying $\mathfrak{X} = \mathfrak{C} - \mathfrak{C}$ (where $\mathfrak{C} - \mathfrak{C}:= \{f-g: f,g\in \mathfrak{C}\}$). \,Suppose that $\mathfrak{T}\colon\mathfrak{X}\to\mathfrak{X}$ is order-preserving, $1$-homogeneous, completely continuous map and for some nonzero $f$, and $M>0$, we have $f\preceq M \mathfrak{T} f$. Then there exists $\lambda>0$ and $\phi\ne 0$ in $\mathfrak{C}$ such that $\mathfrak{T} \phi=\lambda \phi$.
\end{theorem}
Here $\preceq $ denotes the partial ordering in $\mathfrak{X}$ with respect to the cone $\mathfrak{C}$, i.e., $f \preceq g$ if and only if $g-f \in \mathfrak{C}$. Also, we recall that a map $\mathfrak{T}: \mathfrak{X} \to \mathfrak{X}$ is called completely continuous if it is continuous and compact.
Now we state the Aleksandrov-Bakelman-Pucci (ABP) estimate for certain semi-linear differential operator.
\begin{theorem}\label{TAA.2}
Let $v_j\in \mathcal{S}_j$ and $\bar{r}_i(x, u_1, u_2) \leq 0$ for all $(x, u_1, u_2)\in \RR^d\times U_1\times U_2$ and $i, j = 1,2$\,. Suppose that $\phi\in W^{2, p}_{loc}(D)\cap C(\Bar{D})$, $p>d$, satisfies 
\begin{equation}\label{EABP1A}
\mathcal{G}_i^{v_j} \phi \geq f(x) \quad \text{in\ } \{\phi>0\}\cap D\,, \quad\text{with\ } \phi = 0 \text{\ on\ } \partial{D}\,.
\end{equation}
Then the following inequality holds 
\begin{equation*}
\sup_{D} \phi^+ \,\le\, \sup_{\partial D} \phi^+ + \bar{\kappa} \norm{f^-}_{L^d(D)}\,,
\end{equation*}
for some constant positive constant $\bar{\kappa}$\,.
% which depends on $\Lambda$, $M$, $\diam D$, \cref{A2.1}.
\end{theorem}
\begin{proof} 
Since $\bar{b}$ is jointly continuous, $M : = \displaystyle{\sup_{x\in D, u_1\in U_1, u_2\in U_2} \abs{\bar{b}(x,u_1, u_2)} < \infty }$\,. From (\ref{EABP1A}), we deduce that
\begin{equation*}
a_{ij}\frac{\partial^2 \phi}{\partial x_i \partial x_j}(x)+ M\abs{\grad \phi(x)} \,\ge\, f(x) \quad \text{in\ } \{\phi>0\}\cap D\,, \quad\text{with\ } \phi = 0 \text{\ on\ } \partial{D}\,.
\end{equation*} Therefore, the result follows from \cite[Proposition~3.3]{CCKS96}\,.
\end{proof}
We also need the following maximum principle for small domains, which follows form Theorem~\ref{EABP1A}\,.
\begin{lemma}\label{maximumprinciplesmalldomain}
Let $v_j \in {\mathcal S}_j$. Then there exists $\epsilon_0 > 0$ such that if $|D| \leq \epsilon_0$, then any $\varphi \in W^{2, p}_{loc} (D)\cap C(\Bar{D})$
satisfying 
\begin{eqnarray*}
{\mathcal G}^{v_j}_i \varphi  & \geq &  \lambda \varphi,  \  {\rm in} \ D,\\
\varphi &\leq & 0 \ {\rm on} \ \partial D
\end{eqnarray*} for some $\lambda \in \mathbb{R}$, is nonpositive in $D$, where $i\neq j$ and $i, j = 1,2$\,.
\end{lemma}
\begin{proof} 
Take $-(\norm{c}_\infty + |\lambda|) |\varphi| = f$ and $M = \displaystyle{\sup_{D\times U_1\times U_2} \abs{b(x, u_1, u_2)}}$.
Since on $\{\phi > 0\}$,  we have $f^-= (\norm{c}_\infty + |\lambda|)\varphi^+$. Thus, from \ref{EABP1A}, we get
\begin{equation*}
\sup_{D} \varphi^{+} \leq \sup_{\partial D} \varphi^{+} + \hat{K} \| \varphi^{+}\|_{L^d(D)},
\end{equation*}
for some constant $\hat{K} (> 0)$. Now for the choice  $\epsilon_0 = (2 \hat{K})^{-d}$, it follows that for $|D| \leq \epsilon_0$,
\begin{equation*}
\sup_{D} \varphi^{+} \leq \frac{1}{2} \sup_{D} \varphi^{+}\,,
\end{equation*} which is possible only when $\displaystyle{\sup_{D} \varphi^{+} = 0}$\,. Hence $\varphi \leq 0$ in $D$. This completes the proof.
\end{proof}
In view of the above lemma we have the following results. This is useful in establishing simplicity of the generalized principal eigenvalue of smooth bounded domains $D$\,. The proof of the following theorem follows form \cite[Theorem 4.1]{QuaasSirakov}
\begin{theorem}\label{auxillaryresult} Let $v_j \in {\mathcal S}_j$ and
 $\varphi, \psi \in W^{2, p}_{loc}(D) \cap C(\Bar{D}), p \geq d$  satisfies for some
 $\lambda \in \mathbb{R}$
 \begin{eqnarray*}
 {\mathcal G}^{v_j}_i \psi & \leq & \lambda \psi,  \  \psi > 0  \ {\rm in} \ D,\\
 {\mathcal G}^{v_j}_i \varphi & \geq & \lambda \varphi \  {\rm in} \  D,\\
 \varphi & \leq & 0 \ {\rm on} \ \partial D, \  \varphi (x_0) > 0 ,
 \end{eqnarray*}
 for some $x_0 \in D$, then $\psi = t \varphi$ for some $t > 0$, where $i\neq j$ and $i,j = 1,2$\,.
 \end{theorem}
\begin{proof}
 Choose a compact $C \subset D$ such that $|D \setminus C| \leq \epsilon_0$,
 where $\epsilon_0$ is given by Lemma \ref{maximumprinciplesmalldomain}.
Then, following the proof of \cite[Theorem 4.1]{QuaasSirakov} and using the small domain maximum principle as in Lemma~\ref{maximumprinciplesmalldomain} the result follows.
\end{proof}
\end{appendix}
%\section*{Acknowledgement}
% The research of Chandan Pal is partially supported by SERB, India, grant MTR/2021/000307.
%%%%%%%%%%%%%%%%%%%%%%%%%%%%%%%%%%%%%%%%%%%%%%%%%%%%%%%%%%%%%%%%%%%%%%%%%%%%%%


\begin{thebibliography}{99}
\bibitem{AriBorkarGhosh} Arapostathis,  A.,  Borkar,  V.\ S.\  and Ghosh, M.\ K.:   {\it Ergodic Control of Diffusion Processes}, Encyclopedia of Mathematics and its applications 143, Cambridge University Press, (2012).

\bibitem{arapostathis_biswas_saha}
Arapostathis, A., Biswas, A. and Saha, S.: Strict monotonicity of principal eigenvalues of elliptic operators in $\mathbb{R}^{d}$ and risk-sensitive control, {\it J. Math. Pures Appl.}, 124, 169-219, (2019).

\bibitem{arapostathis_biswas}Arapostathis, A. and Biswas, A.: Infinite horizon risk-sensitive control of diffusions without any blanket stability assumptions. {\it Stochastic Process. Appl.}, 128 (5), 1485-1524, (2018).

\bibitem{A18} Arapostathis, A.: A counterexample to a nonlinear version of the Krein-Rutman theorem by R. Mahadevan, {\it Nonlinear Anal.}, 171, 170 - 176, (2018).

\bibitem{AAABSP21} Arapostathis, A., Biswas, A. and Pradhan, S.: On the policy improvement algorithm for ergodic risk-sensitive control, {\it Proceedings of the Royal Society of Edinburgh: Section A Mathematics}, 151(4), 1305 - 1330, (2021).

\bibitem{AAAB2020A} Arapostathis,  A. and Biswas, A.:  A variational formula for risk-sensitive control of diﬀusions in $\RR^d$, {\it SIAM J.Control Optim.}, 58(1), 85 - 103, (2020).

\bibitem{Ba}  Basar, T.:  Nash equilibria of risk-sensitive nonlinear stochastic differential games, {\it J. Optim. Theory Appl.}, 100, 479-498, (1999).

\bibitem{BG} Basu, A.  and Ghosh, M.\ K.:   Risk-sensitive stochastic differential games, {\it Math. Oper. Res.}, 37(3), 437-449, (2012).

\bibitem{RB} Bellman, R.: {\it Dynamic Programming}, Princeton University Press, Princeton, N. J, (1957).

\bibitem{Benes} Bene$\check{s}$, V.\ E.:    Existence of optimal strategies based on specified information, for a class of stochastic decision problems, {\it SIAM J. Control}, 8, 179-188, (1970).

%\bibitem{BensoussanLions} Bensoussan, A.  and  Lions, J.\ L.: {\it Impulse Control and Quasi-Variational Inequalities}, Gauthier-Villars, Paris, (1982).

\bibitem{BN} Bensoussan, A.  and Nagai, H.:   Min-max characterization of a small noise limit on risk-sensitive control, {\it SIAM J. Control Optim.}, 35(4), 1093-1115, (1997).

\bibitem{BFN} Bensoussan, A., Frehse,  J.  and Nagai, H.:  Some results on risk-sensitive control with full observation, {\it Appl. Math. Optim.}, 37(1), 1-41, (1998).

%\bibitem{Berestycki_Rossi} Berestycki, H. and Rossi, L.: Generalizations and Properties of the Principal Eigenvalue of Elliptic Operators in Unbounded Domains {\it Communications on Pure and Applied Mathematics}. 68(6), 1014-1065, (2015).

\bibitem{Anup} Biswas, A.:  Risk sensitive control of diffusions with small running cost. {\it Appl. Math. Optim.}, 64(1), 1-12,  (2011).

\bibitem{AnupBorkarSuresh} Biswas, A., Borkar, V. S.  and Kumar, K. Suresh: Risk-sensitive control with near monotone cost, {\it Appl. Math Optim.},  62, 145-163, (2010). Errata corriege I, ibid, 62(2), 165-167, (2010). Errata corriege II, ibid, 62(3)  435-438, (2010).

\bibitem{Biswas_Saha} Biswas, A. and Saha, S.: Zero-Sum Stochastic Differential Games with risk-sensitive cost. {\it App. Math. Optim.} 81 (2020), 113-140.

%\bibitem{Borkar} Borkar, V. S.:   {\it Optimal Control of Diffusion Processes}, Pitman Research Notes in Mathematics, No. 203, Longman Scientific and Technical, Harlow. UK, (1989).

\bibitem{Borkar-Ghosh} Borkar, V. S.  and Ghosh, M. K.:    Stochastic differential games: occupation measure based approach, {\it J. Optim. Theory Appl.},  73(2), 359-385, (1992). Errata corriege, ibid, 88,  251-252, (1996).

\bibitem{CCKS96} Caffarelli, L., Crandall, M. G., Kocan, M. and Swi\c ech, A.: On viscosity solutions of fully nonlinear equations with measurable ingredients, {\it Comm. Pure Appl. Math.} 49(4), 365-397, (1996).

%\bibitem{CKS} Crandall, M.G., \  Kocan, M., and Swiech, A.: $L^p$-theory of fully nonlinear uniformly
% parabolic equations, {\it Com. PDE}, 25(11)  1997-2053, (2000).

\bibitem{EH} El-Karoui, N. and Hamadene, S.:   BSDE and risk-sensitive control, zero-sum and nonzero-sum game problems of
stochastic functional differential equations, {\it Stochastic Process. Appl.}, 107, 145-169, (2003).

\bibitem{ED} Elliott, R. J.  and Davis, M. H. A.:   Optimal play in a stochastic differential game, {\it SIAM J. Control Optim.}, 19(4), 543-554, (1981).

\bibitem{Fan} Fan, K.: Fixed-point and minimax theorems in locally convex topological linear spaces. {\it Proc. Nat. Acad. Sc.}, 38,  121-126, (1952).

\bibitem{FH} Fleming, W.\ H.  and Hern\'{a}ndez-Hern\'{a}ndez, D.:   On
the value of stochastic differential games, {\it Comm. Stoch. Anal.},
5, 341-351, (2011).

\bibitem{FM} Fleming, W.\ H.  and McEneaney, W.\ M.:   Risk-sensitive Control on an infinite time horizon, {\it SIAM J. Control Optim.} 33(6), 1881-1915, (1995).

\bibitem{MKGSP22A}Ghosh, M.\ K. and Pradhan, S.: A nonzero-sum risk-sensitive stochastic differential game in the orthant, {\it Mathematical Control \& Related Fields}, 12(2), 343-370, (2022).

\bibitem{MKGSP20A}Ghosh, M.\ K. and Pradhan, S.: Ergodic risk-sensitive stochastic diﬀerential games with reﬂecting diﬀusions in abounded domain,{\it Stochastic Analysis and Applications}, 39(5), 819-841, (2020).

\bibitem{GilbargTrudinger} Gilbarg, D. and Trudinger, N.S., \textit{Elliptic Partial Differential Equations of Second Order},  Classics in Mathematics, Reprint of 1998 Edition, Springer 2001.

\bibitem{J} Jacobson, D. H.: Optimal stochastic linear systems with exponential performance criteria and their relation to deteministic differential games, {\it IEEE Trans. Automat. Control.}  AC-18, 124-131, (1973).

\bibitem{MR} Menaldi, J-L.  and Robin, M.:  Remarks on risk-sensitive control problems, {\it Appl. Math. Optim.},  52(3), 297-310, (2005).

%\bibitem{MS} Milgrom, P.  and Segal, I.:   Envelope theorems for arbitrary choice sets. {\it Econometrica}, 70(2),  583-601, (2002).


\bibitem{Na} Nagai, H.:  Optimal strategies for risk-sensitive portfolio optimization problems for general factor models, {\it SIAM J. Control Optim.}, 41(6),  1779-1800, (2003).

\bibitem{No} Nowak, A. S.:   Notes on risk-sensitive Nash equilibria. Advances in dynamic games,  {\it Ann. Internat. Soc. Dynam. Games}, 7, Birkhauser Boston, Boston, MA, 95-109, (2005).

\bibitem{ON84} Okada, N.: On the Banach-Saks Property, {\it Proc. Japan Acad.}, 60(A), 246 - 248, (1984).

\bibitem{SP21} Pradhan, S.: Risk-Sensitive Ergodic Control of Reflected Diffusion Processes in Orthant, {\it Appl Math Optim}, 83, 1739-1764, (2021). https://doi.org/10.1007/s00245-019-09606-w

\bibitem{SP21A} Pradhan, S.: Risk-sensitive zero-sum stochastic differential game for jump–diffusions, {\it Systems \& Control Letters}, 157, 105033, (2021). https://doi.org/10.1016/j.sysconle.2021.105033

\bibitem{QuaasSirakov} Quaas, A. and Sirakov, B., Principal eigenvalues and the Dirichlet problem for fully nonlinear elliptic operators. \textit{Adv. Math.} 218(1), 105-135, 2008.

\bibitem{Ru}  Runolfsson, T.:  Robust control of discrete-time hybrid systems with uncertain modal dynamics, {\it Math. Probl. Eng.}, 5(6), 459-478, (2000).

\bibitem{Va} Varaiya, P.:   $N$-player stochastic differential games, {\it SIAM J. Control Optim.}, 14(3),  538-545, (1976).

%\bibitem{Veretennikov}  Veretennikov, A. Yu.:  On strong solutions and explicit formulas for solutions of stochastic integral equations, {\it Math USSR-Sb.}, 39, 387-403, (1981).

%\bibitem{Warga} Warga, J.:  , Functions of relaxed controls, {\it SIAM J. Control}, 5, 628-641, (1967).


\bibitem{PW} Whittle, P.:  {\it Risk-Sensitive Optimal Control}, Wiley-Interscience Series in Systems and Optimization, John Wiley $ \& $ Sons, Ltd., Chichester, (1990).
\end{thebibliography}
\end{document}